\documentclass[12pt]{amsart}

\usepackage{fancyhdr}
\usepackage{enumitem} 
\newlist{condenum}{enumerate}{1}
\usepackage{xcolor} 
\usepackage{graphicx} 
\usepackage{hyperref} 

\usepackage{amsmath}
\allowdisplaybreaks
\usepackage{amsthm}
\usepackage{amssymb}
\usepackage{mathtools}
\usepackage{mathabx}
\usepackage{bbm}
\usepackage{shuffle}
\usepackage[mathscr]{eucal}
\usepackage{stmaryrd} 
\usepackage{makecell}

\usepackage{tikz}
\usetikzlibrary{patterns, arrows, matrix, positioning, calc}
\usepackage{tikz-cd}

\usepackage{xcolor}
\usepackage[colorinlistoftodos]{todonotes}

\newtheorem{thm}[equation]{Theorem}

\newtheorem{prop}[equation]{Proposition}
\newtheorem{lem}[equation]{Lemma}
\newtheorem{cor}[equation]{Corollary}
\theoremstyle{definition}

\theoremstyle{remark}
\newtheorem{rem}[equation]{Remark}

\numberwithin{equation}{section}
\newenvironment{introthm}[1]{
  \renewcommand\theequation{#1}
  \thm
}{\endthm}
\newenvironment{introcor}[1]{
  \renewcommand\theequation{#1}
  \cor
}{\endcor}

\newcommand{\CC}{\mathbb{C}}

\newcommand{\QQ}{\mathbb{Q}}
\newcommand{\ZZ}{\mathbb{Z}}

\def\mydefb#1{\expandafter\def\csname #1#1#1\endcsname{\mathcal{#1}}}
\def\mydefallb#1{\ifx#1\mydefallb\else\mydefb#1\expandafter\mydefallb\fi}
\mydefallb ABCDEFGHIJKLMNOPQRSTUVWXYZ\mydefallb

\renewcommand{\epsilon}{\varepsilon}
\newcommand{\spanning}{\operatorname{-span}}

\renewcommand{\setminus}{-}

\newcommand{\TL}{\mathsf{TL}}

\newcommand{\Sym}{\mathrm{Sym}}
\newcommand{\QSym}{\mathrm{QSym}}
\newcommand{\QSV}{\mathrm{QSV}}
\newcommand{\NCP}{\mathrm{NCP}}

\newcommand{\EP}{E_{pos}}
\newcommand{\EV}{E_{val}}

\newcommand{\edge}[2]{\tikz[scale = 0.75, baseline = -0.2cm]{
\draw (0.1, 0) node[inner sep = 0ex] (a) {};
\node[left, xshift = 0.1cm, yshift = -0.1cm] at (a) {$\scriptstyle #1$};
\draw (0.9, 0)  node[inner sep = 0ex] (b) {};
\node[right, xshift = -0.1cm, yshift = -0.1cm] at (b) {$\scriptstyle #2$};
\draw[thick] (a) to[out = 35, in = 145] (b);}}

\title{The excedance quotient of the Bruhat order, Quasisymmetric Varieties and Temperley-Lieb algebras}
\author{ Nantel Bergeron and Lucas Gagnon}
\address{Dept. of Math. and Stat.\\ York  University\\ To\-ron\-to, Ontario M3J 1P3\\ CANADA}
\email{bergeron@yorku.ca, lgagnon@yorku.ca}
\thanks{This work is supported in part by York Research Chair and NSERC.
This paper originated in a working session at the Algebraic
Combinatorics Seminar at Fields Institute}
\date{}


\begin{document}

\maketitle
\begin{abstract} 
Let $R_n=\mathbb{Q}[x_1,x_2,\ldots,x_n]$ be the ring of polynomial in $n$ variables and consider the ideal $\langle \mathrm{QSym}_{n}^{+}\rangle\subseteq R_n$ generated by quasisymmetric polynomials without constant term. 
It was shown by J.~C.~Aval, F.~Bergeron and  N.~Bergeron that $\dim\big(R_n\big/\langle \mathrm{QSym}_{n}^{+} \rangle\big)=C_n$ the $n$th Catalan number. 
In the present work, we explain this phenomenon by defining a set of permutations $\mathrm{QSV}_{n}$ with the following properties: first, $\mathrm{QSV}_{n}$ is a basis of the Temperley--Lieb algebra $\mathsf{TL}_{n}(2)$, and second, when considering $\mathrm{QSV}_{n}$ as a collection of points in $\mathbb{Q}^{n}$, the top-degree homogeneous component of the vanishing ideal $\mathbf{I}(\mathrm{QSV}_{n})$ is $\langle \mathrm{QSym}_{n}^{+}\rangle$.

Our construction has a few byproducts which are independently noteworthy.  
We define an equivalence relation $\sim$ on the symmetric group $S_{n}$ using weak excedances and show that its equivalence classes are naturally indexed by noncrossing partitions.  
Each equivalence class is an interval in the Bruhat order between an element of $\mathrm{QSV}_{n}$ and a $321$-avoiding permutation.  
Furthermore, the Bruhat order induces a well-defined order on $S_{n}\big/\!\!\sim$.  
Finally, we show that any section of the quotient $S_{n}\big/\!\!\sim$ gives an (often novel) basis for $\mathsf{TL}_{n}(2)$.  

\end{abstract}

\section{Introduction}

Quasisymmetric functions originate in the work of Stanley~\cite{StanleyQSym}, where they appear as enumeration series for $P$-partitions.  
Later, Gessel~\cite{Gessel} gave a more algebraic treatment of the ring $\QSym$ spanned by all quasisymmetric functions, establishing a beautiful analogy with the classical ring of symmetric functions $\Sym$. 
The importance of $\QSym$ has continued to increase:~\cite{ABS} established $\QSym$ as a universal setting for enumerative combinatorial invariants, and in recent years quasisymmetric functions have been at the center of a number of research programs (many examples can be found in~\cite{Grinberg,LMvW, Mason} and references therein).

There is also a striking similarity between quasisymmetric functions and the invariant theory of finite reflections groups.  
Chevalley's theorem states that each finite reflection group $W$ acts naturally on a polynomial ring $R$, and the quotient of $R$ by the ideal $\langle R_{+}^{W} \rangle$ generated by positive degree invariants is isomorphic to the regular module of $W$; see~\cite[Chapter 3]{Humphrey}.  
Similarly, the quasisymmetric polynomials $\QSym_{n}$ in $R_{n} = \QQ[x_{1}, \ldots, x_{n}]$ are the invariants of an action of  Temperley--Lieb algebra $\TL_{n}(2)$ on $R_{n}$  introduced by Hivert~\cite{Hivert}.  
Writing $\langle \QSym_{n}^{+} \rangle$
for ideal generated by the positive degree quasisymmetric polynomials,~\cite{AB,ABB} show that the dimension of the coinvariant space $R_{n}\big/\langle \QSym_{n}^{+} \rangle$ and $\TL_{n}(2)$ agree: both are the $n$th Catalan number $C_{n}$.  
Since $\TL_{n}(2)$ shares many nice properties with reflection groups, one might expect a Chevalley-type theorem from this coincidence, but there is no obvious $\TL_{n}(2)$-action on $R_{n}\big/\langle \QSym_{n}^{+} \rangle$:  Hivert's action is not multiplicative and $\langle \QSym_{n}^{+} \rangle$ is not a $\TL_{n}(2)$-submodule.  

Motivated by the discussion above, we revisit two modules which afford the left regular representation of the symmetric group $S_{n}$:
\begin{enumerate}[itemsep = 1ex]
\item the quotient $R_{n}\big/\langle \Sym_{n}^{+} \rangle$ of the polynomial ring $R_{n} = \QQ[x_{1}, \ldots, x_{n}]$ by the ideal generated by positive-degree symmetric polynomials $\Sym_{n}^{+}$, and

\item the coordinate ring $R_{n} / \mathbf{I}(S_{n})$ for the vertices of the regular permutohedron in $\QQ^{n}$, which are the points $(\sigma_{1}, \ldots, \sigma_{n})$ for each $\sigma \in S_{n}$.

\end{enumerate}
Module (1) is a famous case of Chevalley's theorem: the $S_{n}$-invariants of $R_{n}$ are the symmetric polynomials, and $R_{n}\big/\langle \Sym_{n}^{+} \rangle$ is the $S_{n}$ coinvariant ring.  
On the other hand, module (2) comes from the left multiplicative action of $S_{n}$ on the permutahedron realized on the coordinate ring $R_{n} / \mathbf{I}(S_{n})$  where $\mathbf{I}(S_{n})$ is the vanishing ideal. 
However, as seen in the work of Garsia and Procesi~\cite{GP} and reference therein, a careful inspection reveals that these modules determine one another!  
Consider the ideal
\[
I_{n} = \langle f(x_{1}, \ldots, x_{n}) - f(1, \ldots, n) \;|\; f \in \Sym_{n}^{+} \rangle \subseteq \mathbf{I}(S_{n}).
\]
For each $f \in R_{n}$, let $\mathsf{h}(f)$ denote the top-degree homogeneous component of $f$, and for any ideal $I$ in $R_{n}$ write $\mathsf{gr}(I) = \langle \mathsf{h}(f) \;|\; f \in I \rangle$.  
Then $\mathsf{gr}(I_{n}) \supseteq \langle \Sym_{n}^{+} \rangle$, and Gr\"{o}bner basis theory gives a linear isomorphism $R_{n}\big/\mathsf{gr}(I_{n})  \cong R_{n}\big/I_{n}$; see Section~\ref{sec:Grobner}.  We therefore have
\[
|S_{n}| = \dim\big(R_{n}\big/\langle \Sym_{n}^{+} \rangle\big) \ge \dim\big(R_{n}\big/\mathsf{gr}(I_{n})\big) = \dim\big(R_{n}\big/I_{n}) \ge \dim(R_{n}\big/\mathbf{I}(S_{n})\big) = |S_{n}|,
\]
so that $I_{n} = \mathbf{I}(S_{n})$ and $\mathsf{gr}(I_{n}) = \langle \Sym_{n}^{+} \rangle$, and $R_{n}\big/\langle \Sym_{n}^{+} \rangle \cong R_{n}\big/\mathbf{I}(S_{n})$ as vector spaces.  
This isomorphism respects the $S_{n}$-action on each quotient: both $\mathbf{I}(S_{n})$ and $\langle \Sym_{n}^{+} \rangle$ are fixed spaces for the standard $S_{n}$-action on $R_{n}$, and this action coindices with the action on points for $R_{n}\big/\mathbf{I}(S_{n})$.  
Thus, we have an $S_{n}$-module isomorphism $R_{n}\big/\langle \Sym_{n}^{+} \rangle \cong R_{n}\big/\mathbf{I}(S_{n})$, though the left hand side has a natural grading and the right hand side does not.

In the present paper, we attempt to apply some of the ideas presented to quasisymmetric functions and Temperley--Lieb algebras. 
It is known  that $\langle \Sym_{n}^{+} \rangle \subseteq \langle \QSym_{n}^{+} \rangle$, and that there is a surjective algebra homomorphism $\phi: \CC S_{n} \to \TL_{n}(2)$. 
Guided by these relationships, we search for a subset $\QSV_{n} \subseteq S_{n} \subseteq \QQ^{n}$ which satisfies:
\begin{enumerate}[itemsep = 1ex, label=(\roman*)]
\item $|\QSV_{n}| = C_{n}$,

\item the image $\phi(\QSV_{n})$ is a basis of $\TL_{n}(2)$, and

\item considering the vanishing ideal $\mathbf{I}(\QSV_{n})$, we have $\mathsf{gr}\big(\mathbf{I}(\QSV_{n})\big) = \langle \QSym_{n}^{+} \rangle$.

\end{enumerate}
Assuming such a set exists, one can define an action of $\TL_{n}(2)$ on the space $R_{n}\big/\langle \QSym_{n}^{+} \rangle$ using Gr\"{o}bner basis theory and the multiplication constants for the basis obtained from $\QSV_{n}$.  
However, $\QSV_{n}$ is not readily found: it took several years of computer exploration to find a list of candidates for small values of $n$.  
We have now found it, along with a number of remarkable properties that should be of interest to the wider community.

The set $\QSV_{n} \subseteq S_{n}$ is defined in Section~\ref{sec:QSV}.  
After our initial discovery, we noticed that the cycle structures of these permutations determine a noncrossing partition, tying them to a more general story about noncrossing partitions and the symmetric group~\cite{Baine} (see also~\cite{McCammond}).  
For example, writing $Q_{\lambda}$ to denote the element of $\QSV_{n}$ indexed by the partition $\lambda$,
\[
\lambda = \begin{tikzpicture}[scale = 0.75, baseline = 0.75*-0.2]
\foreach \x in {1, ..., 7}{\draw[fill] (\x - 1, 0) node[inner sep = 2pt] (\x) {$\scriptstyle \x$};}
\foreach \i\j in {5/6, 3/5,2/7}{\draw[thick] (\i) to[out = 35, in = 145] (\j);}
\end{tikzpicture}
\qquad\text{corresponds to}\qquad
Q_{\lambda} = (1)(72)(653)(4).
\]
Through this connection,~\cite{GobetWilliams} and ~\cite{Zinno} have studied bases of general Temperley--Lieb algebras which specialize to $\phi(\QSV_{n})$ for $\TL_{n}(2)$, so only condition (iii) remains.

Our initial attempts to prove condition (ii) led us to an exciting discovery about how $\QSV_{n}$ sits in $S_{n}$.  
In Section~\ref{sec:excedance} we define an equivalence relation $\sim$ on $S_{n}$ using the weak excedance set of a permutation and its inverse.  
We call the equivalence classes of $S_{n}\big/\!\!\sim$ \emph{excedance classes}, and show that each noncrossing partition $\lambda$ bijectively determines an excedance class $\CCC_{\lambda}$.  
Surprisingly, the Bruhat order induces a well-defined quotient order on excedance classes.  
In the following, $\preceq$ denotes the order on noncrossing partitions which is dual to Young's lattice, described further in Section~\ref{sec:bruhatballot}.

\begin{introthm}{\ref{thm:excedancequotient}}
Writing $\le$ for the relation on excedance classes $S_{n}\big/\!\! \sim$ induced by the Bruhat order, $\CCC_{\lambda} \le \CCC_{\mu}$ if and only if $\lambda \preceq \mu$.
\end{introthm}

A similar result is given by~\cite{GobetWilliams} for the set $\QSV_{n}$ as a sub-poset of the Bruhat order (see Section~\ref{sec:bruhatballot}), which simplifies our proof of Theorem~\ref{thm:excedancequotient}.  This leads to a kind of duality between these sub- and quotient orders of the Bruhat poset.

\begin{introcor}{\ref{cor:interval}}
Each excedance class $\CCC_{\lambda}$ is an interval in the Bruhat order, with upper bound $Q_{\lambda}\in \QSV_n$ and lower bound given by a $321$-avoiding permutation
\end{introcor}

The combinatorics of excedance classes are very rich, and there is much left to explore.   

In Section~\ref{sec:TLbasis}, we use excedance classes of $S_{n}$ to explore bases of $\TL_{n}(2)$.  Using results of~\cite{GobetWilliams} and~\cite{Zinno}, our Theorem~\ref{thm:TLbasis} restates the fact that $\QSV_{n}$ satisfies condition (ii) above.  However, we also prove more general (and novel) theorem about bases of $\TL_{n}(2)$ coming from the surjection $\phi: \CC S_{n} \to \TL_{n}(2)$.

\begin{introthm}{\ref{thm:TLbases}}
Let $n \ge 0$ and for each noncrossing partition $\lambda$ of size $n$, fix an element $w_{\lambda} \in \CCC_{\lambda}$. Then the set $\{\phi(w_{\lambda}) \;|\; \text{noncrossing partitions $\lambda$}\}$ is a basis of $\TL_{n}(2)$.
\end{introthm}

Finally, in Section~\ref{sec:QSymvanish} we show that the set $\QSV_{n}$ satisfies condition (iii) above.  
The space of positive-degree quasisymmetric polynomials $\QSym_{n}$ has a homogeneous basis of monomial quasisymmetric functions $M_{\alpha}$ indexed by the compositions $\alpha \vDash d$ of a positive integer $d > 0$ with length $\ell(\alpha) \le n$.  
For each such composition $\alpha$, we construct a nonhomogeneous polynomial $P_{\alpha} \in R_{n}$ for which $\mathsf{h}(P_{\alpha}) = M_{\alpha}$ and show the following.

\begin{introthm}{\ref{thm:vanishingQSV}}
The ideal $\langle P_{\alpha} \;|\; \text{$\alpha \vDash d $ with $d> 0$ and $\ell(\alpha) \le n$} \rangle \subseteq R_n$ is the vanishing ideal $\mathbf{I}(\QSV_n)$ and 
\[
 \langle \QSym_{n}^{+} \rangle = \mathsf{gr}\big(\mathbf{I}(\QSV_{n})\big).
\]
\end{introthm}

From this, we obtain a linear isomorphism
\[
R_{n}\big/\mathbf{I}(\QSV_{n}) \cong R_{n}\big/\langle \QSym_{n}^{+} \rangle.
\]
Future work will explore the module structure on $R_{n}\big/\langle \QSym_{n}^{+} \rangle$ implicit in this result.

\subsection{Acknowledgement} 
We thank the participants of the Algebraic Combinatorics working seminar at the Fields Institute: K.~Chan, F.~G\'elinas, A.~Lao, A.~A.~Lazzeroni, C.~McConnell, Y.~Solomon, F.~Soltani,  N.~Wallace and 
M.~ Zabrocki.
We are also grateful to R.~Gonz\'alez~D'L\'eon for his  inspiring discussions.

\section{Preliminaries}

In this section we will recall some definitions and preliminary results about noncrossing partitions (Section~\ref{sec:ncp}) and partial orders on the symmetric group (Section~\ref{sec:bruhat}).

\subsection{Noncrossing partitions}
\label{sec:ncp}

Let $n$ be a nonnegative integer.  A \emph{noncrossing partition} of size $n$ is a diagram $\lambda$ consisting of:
\begin{enumerate}
\item the positive integers $1, \ldots, n$, placed from left to right along a horizontal axis; and

\item a set of left-to-right arcs $\edge{i}{j} = (i, j)$, $i < j$ drawn above the axis with no intersections or coterminal points: $\lambda$ contains no pair $\edge{i}{k}, \edge{j}{l}$ with $i \le j < k \le l$.

\end{enumerate}
For example,
\begin{equation}
\label{eq:noncrossingpartitionexample}
\lambda = \begin{tikzpicture}[scale = 0.75, baseline = 0.75*-0.2]
\foreach \x in {1, ..., 7}{\draw[fill] (\x - 1, 0) node[inner sep = 2pt] (\x) {$\scriptstyle \x$};}
\foreach \i\j in {5/6, 3/5,2/7}{\draw[thick] (\i) to[out = 35, in = 145] (\j);}
\end{tikzpicture}
\end{equation}
is a noncrossing partition of size $7$ containing three arcs: $\edge{2}{7}$, $\edge{3}{5}$, and $\edge{5}{6}$.

Considering a noncrossing partition $\lambda$ as an (undirected) graph, the connected components of $\lambda$ give a partition of the set $[n] = \{1, \ldots, n\}$, which is the origin of the term.  For example, the noncrossing partition shown in Equation~\eqref{eq:noncrossingpartitionexample} corresponds to the set partition $\big\{ \{1\}, \{2, 7\}, \{3, 5, 6\}, \{4\}  \big\}$.  Let
\[
\NCP_{n} = \{ \text{noncrossing partitions of size $n$} \}.
\]
The number of noncrossing partitions of size $n$ is the $n$th Catalan number, $C_{n} = \frac{1}{n+1}\binom{2n}{n}$~\cite[Exercise 6.19 pp]{Stanley}.

Given an arc $ \edge{i}{j} \in \lambda$, say that $i$ is the \emph{left endpoint} and $j$ is the \emph{right endpoint}, and let
\[
\lambda^{+} = \{ i \in [n] \;|\; \text{$i$ is a left endpoint in $\lambda$} \}
\]
and
\[
\lambda^{-} = \{ i \in [n] \;|\; \text{$i$ is a right endpoint in $\lambda$} \}.
\]
For example, with the noncrossing partition $\lambda$ in~\eqref{eq:noncrossingpartitionexample}, $\lambda^{+} = \{2, 3, 5\}$ and $\lambda^{-} = \{5, 6, 7\}$.  The arcs in $\lambda$ give a bijection between the sets $\lambda^{+}$ and $\lambda^{-}$, so that
\[
|\lambda^{+}| = |\lambda^{-}|.
\]

The following lemma is classic in the literature about non-crossing partitions (for example, see~\cite{Stanley}). 
As it plays an important role in our main construction, we have included a proof for the sake of exposition.
\begin{lem}
\label{lem:noncrossingpartitionproperty}
Each noncrossing partition $\lambda$ of size $n$ is uniquely determined by the sets $\lambda^{+}$ and $\lambda^{-}$.  Moreover, given two subsets $L$ and $R$ of $[n]$ with equal size, the inequalities
\[
\big|[k-1] \cap L\big| \ge \big|[k] \cap R\big| \qquad\text{for all $k \ge 1$}
\]
hold if and only if we have $L = \lambda^{+}$ and $R = \lambda^{-}$ for some noncrossing partition $\lambda$ of size $n$.
\end{lem}
\begin{proof}
Given $L$ and $R$, draw the elements of $[n]$ along the horizontal axis, increasing from left to right, and draw a half-arc starting at each vertex in $L$ and and half-arc ending at each vertex in $R$.  For example, with $n = 8$, $L = \{2, 4, 7\}$, and $R = \{5, 6, 8\}$, the resulting diagram is
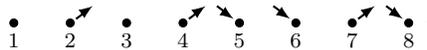
\[
\begin{tikzpicture}[scale = 0.75, baseline = 0.75*-0.2]
\foreach \x in {1, ..., 8}{\draw[fill] (\x - 1, 0) circle (2pt) node[inner sep = 2pt] (\x) {};}
\foreach \x in {1, ..., 8}{\node[below] at (\x) {$\scriptstyle \x$};}
\foreach \i\j in {2/6, 4/5, 7/8}{
	\draw[thick, -latex] (\i) to[out = 35, in = 215] ($ (\i) + (0.4, 0.3) $);
	\draw[thick, -latex] ($ (\j) + (-0.4, 0.3) $) to[out = -35, in = 145] (\j);}
\end{tikzpicture}.
\]
These half-arcs determine a unique noncrossing partition, which can be obtained by recursively connecting pairs of half-edges which have no incomplete edges between them. 
This is essentially the same process as matching open and closed parenthesis, the only difference being that our half-edges sit in prescribed positions.  
The assumed condition on the sets $L$ and $R$ ensures that each ending half-arc in $R$ will have a starting half arc to match with.  
The noncrossing condition follows from our choice of connections: a crossing would imply some connection was made between half-edges with an incomplete edge between them.  
Continuing the preceding example, this algorithm successively gives the diagrams
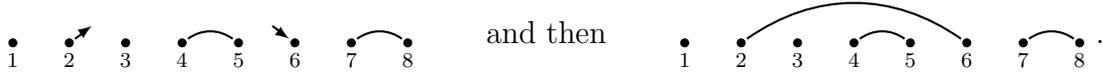
\[
\begin{tikzpicture}[scale = 0.75, baseline = 0.75*-0.2]
\foreach \x in {1, ..., 8}{\draw[fill] (\x - 1, 0) circle (2pt) node[inner sep = 2pt] (\x) {};}
\foreach \x in {1, ..., 8}{\node[below] at (\x) {$\scriptstyle \x$};}
\foreach \i\j in {2/6}{
	\draw[thick, -latex] (\i) to[out = 35, in = 215] ($ (\i) + (0.4, 0.3) $);
	\draw[thick, -latex] ($ (\j) + (-0.4, 0.3) $) to[out = -35, in = 145] (\j);}
\foreach \i\j in {4/5, 7/8}{\draw[thick] (\i) to[out = 35, in = 145] (\j);}
\end{tikzpicture}
\qquad\text{and then}\qquad
\begin{tikzpicture}[scale = 0.75, baseline = 0.75*-0.2]
\foreach \x in {1, ..., 8}{\draw[fill] (\x - 1, 0) circle (2pt) node[inner sep = 2pt] (\x) {};}
\foreach \x in {1, ..., 8}{\node[below] at (\x) {$\scriptstyle \x$};}
\foreach \i\j in {2/6, 4/5, 7/8}{\draw[thick] (\i) to[out = 35, in = 145] (\j);}
\end{tikzpicture}.
\]

For the converse, if $L = \lambda^{+}$ and $R = \lambda^{-}$ for a nonnesting partition $\lambda$, the given inequalities must hold: the left endpoint of each arc must be strictly less than the right endpoint.
\end{proof}

\subsection{Permutations and the Bruhat order}
\label{sec:bruhat}

Let $S_n$ denote the group of permutations of $[n]$.  We represent elements of $S_{n}$ either by using the standard one- and two-line notations or as a product of cycles.  We also write $\ell$ for the length function, so that for $w \in S_{n}$, $\ell(w)$ is the number of inversions of $w$:
\[
\ell(w) = |\{(i, j) \;|\; \text{$1 \le i < j \le n$ and $w_{i} > w_{j}$}\}|.
\]

The \emph{Bruhat order} on $S_{n}$ is the partial order generated by the relation
\[
v < w \qquad\text{if and only if} \qquad \text{$wv^{-1}$ is a transposition $(i\,j)$ and $\ell(v) < \ell(w)$}.
\]
This order is ubiquitous in the study of $S_{n}$ and related objects (for examples, see the book~\cite{BjornerBrenti}).

It is difficult to study the Bruhat order using only the generating relations above, so we will use the so-called \emph{tableau criterion} to determine its other relations.  For $w \in S_{n}$ and $1 \le k \le n$, let $\TTT_{k}(w)$ be the increasing re-arrangement of the set $\{w_{1}, w_{2}, \ldots, w_{k}\}$, and let $\TTT(w)$ be the array with rows $\TTT_{n}(w)$, $\TTT_{n-1}(w)$, and on to $\TTT_{1}(w)$.  For example, $\TTT_{4}(52314) = 1\,2\,3\,5$, and
\[
\TTT(52314) = 
\begin{array}{ccccc} 
1 & 2 & 3 & 4 & 5 \\ 
1 & 2 & 3 & 5 \\ 
2 & 3 & 5 \\ 
2 & 5 \\ 
5
\end{array}
\qquad\text{and}\qquad
\TTT(41235) = \begin{array}{ccccc} 
1 & 2 & 3 & 4 & 5 \\ 
1 & 2 & 3 & 4 \\ 
1 & 2 & 4 \\ 
1 & 4 \\ 
4
\end{array}.
\]

\begin{prop}[{\cite[Theorem 2.6.3]{BjornerBrenti}}]
\label{TableauCriterion}
For $v, w \in S_{n}$, we have $v \le w$  if and only if each entry of $\TTT(v)$ is less than or equal to to corresponding entry of $\TTT(w)$.
\end{prop}

For example, by considering the diagrams $\TTT(52314)$ and $\TTT(41235)$ shown above, the tableau criterion allows us to deduce that $52314 \le 41235$.

\section{The set $\QSV_{n}$}
\label{sec:QSV}

In this section we define the set $\QSV_{n} \subseteq S_n$ and establish its elementary properties.  Our treatment is essentially the standard on in the literature on noncrossing partitions and is originally due to~\cite{Baine}.  In Section~\ref{sec:bruhatballot}, we turn to the restriction of the Bruhat order to $\QSV_{n}$ and recall the combinatorial description of this order from~\cite{GobetWilliams}. 

Let $\lambda$ be a noncrossing partition of size $n$ and recall the sets $\lambda^{+}$ and $\lambda^{-}$ from Section~\ref{sec:ncp}.  
Define a permutation $Q_{\lambda} \in S_{n}$ by 
\[
Q_{\lambda}(j) = \begin{cases} i & \text{if $j \in \lambda^{-}$ and $ \edge{i}{j} \in \lambda$} \\
k & \text{if $j \notin \lambda^{-}$ and $k$ is the largest element connected to $i$ in $\lambda$}
\end{cases}
\]
Thus, $Q_{\lambda}$ sends each $j \in [n]$ to its leftward neighbor in $\lambda$, if such a neighbor exists, and otherwise sends $j$ to the rightmost element of its connected component.  

Let
\[
\QSV_{n} = \{Q_{\lambda} \;|\; \lambda \in \NCP_{n} \}.
\]
For example, the elements of $\QSV_{3}$ are:
\[
Q_{\begin{tikzpicture}[scale = 0.35, baseline = 0.35*-0.2]
\foreach \x in {1, ..., 3}{\draw[fill] (\x - 1, 0) circle (2pt) node[inner sep = 2pt] (\x) {};}
\foreach \x in {1, ..., 3}{\node[below] at (\x) {$\scriptstyle \x$};}
\foreach \i\j in {1/3}{\draw[thick] (\i) to[out = 35, in = 145] (\j);}
\end{tikzpicture}} = 321, \qquad
Q_{\begin{tikzpicture}[scale = 0.35, baseline = 0.35*-0.2]
\foreach \x in {1, ..., 3}{\draw[fill] (\x - 1, 0) circle (2pt) node[inner sep = 2pt] (\x) {};}
\foreach \x in {1, ..., 3}{\node[below] at (\x) {$\scriptstyle \x$};}
\foreach \i\j in {1/2, 2/3}{\draw[thick] (\i) to[out = 35, in = 145] (\j);}
\end{tikzpicture}} = 312, \qquad
Q_{\begin{tikzpicture}[scale = 0.35, baseline = 0.35*-0.2]
\foreach \x in {1, ..., 3}{\draw[fill] (\x - 1, 0) circle (2pt) node[inner sep = 2pt] (\x) {};}
\foreach \x in {1, ..., 3}{\node[below] at (\x) {$\scriptstyle \x$};}
\foreach \i\j in {1/2}{\draw[thick] (\i) to[out = 35, in = 145] (\j);}
\end{tikzpicture}} = 213, 
\]
\[
Q_{\begin{tikzpicture}[scale = 0.35, baseline = 0.35*-0.2]
\foreach \x in {1, ..., 3}{\draw[fill] (\x - 1, 0) circle (2pt) node[inner sep = 2pt] (\x) {};}
\foreach \x in {1, ..., 3}{\node[below] at (\x) {$\scriptstyle \x$};}
\foreach \i\j in {2/3}{\draw[thick] (\i) to[out = 35, in = 145] (\j);}
\end{tikzpicture}} = 132, \qquad\text{and}\qquad
Q_{\begin{tikzpicture}[scale = 0.35, baseline = 0.35*-0.2]
\foreach \x in {1, ..., 3}{\draw[fill] (\x - 1, 0) circle (2pt) node[inner sep = 2pt] (\x) {};}
\foreach \x in {1, ..., 3}{\node[below] at (\x) {$\scriptstyle \x$};}
\end{tikzpicture}} = 123.
\]

\begin{lem}
\label{lem:QSVcycles}

Let $\lambda$ be a noncrossing partition of size $n$ with connected components $C_{1}, \ldots, C_{s}$, and for $1 \le r \le s$ enumerate $C_{r}$ in increasing order as $\{c_{r, 1} < c_{r, 2} < \cdots < c_{r, |C_{r}|}\}$. 
We then have the disjoint cycle decomposition
\[
Q_{\lambda} = \prod_{r = 1}^{s} (c_{r, |C_{r}|} \cdots c_{r, 2} c_{r, 1} ).
\]
\end{lem}
\begin{proof}
The statement follows from the definition of $Q_{\lambda}$ above: $[n] \setminus \lambda^{-} = \{c_{r, 1} \;|\; 1 \le r \le s\}$, so that $Q_{\lambda}(c_{r, 1}) =  c_{r, |C_{r}|}$ for each $r$, and for $1 < i \le |C_{r}|$, we have $c_{r, i} = c_{r, i-1}$.
\end{proof}

For example, when $\lambda$ has a single connected component, $Q_{\lambda}$ is a single cycle: with
\[
\lambda = \begin{tikzpicture}[scale = 0.75, baseline = 0.75*-0.2]
\foreach \x in {1, ..., 7}{\draw[fill] (\x - 1, 0) node[inner sep = 2pt] (\x) {$\scriptstyle \x$};}
\foreach \i\j in {1/2, 2/3, 3/4, 4/5, 5/6, 6/7}{\draw[thick] (\i) to[out = 35, in = 145] (\j);}
\end{tikzpicture}
\qquad\text{we have}\qquad
Q_{\lambda} = (7654321) = 7123456.
\]
Considering the noncrossing partition shown in Equation~\eqref{eq:noncrossingpartitionexample} gives a more complicated example: with
\[
\lambda = \begin{tikzpicture}[scale = 0.75, baseline = 0.75*-0.2]
\foreach \x in {1, ..., 7}{\draw[fill] (\x - 1, 0) node[inner sep = 2pt] (\x) {$\scriptstyle \x$};}
\foreach \i\j in {5/6, 3/5,2/7}{\draw[thick] (\i) to[out = 35, in = 145] (\j);}
\end{tikzpicture}
\qquad\text{we have}\qquad
Q_{\lambda} = (1)(72)(653)(4) = 1764352.
\]

\begin{rem}
\label{rem:QSVnoncrossing}
Given any $n$-cycle $c \in S_{n}$,~\cite{Baine} gives a bijection between $\NCP_{n}$ and the interval between the identity and $c$ in the absolute order on $S_{n}$.  Our construction of $\QSV_{n}$ realizes this bijection for the $n$-cycle $c = (n\cdots 21)$.  
\end{rem}

\subsection{The Bruhat order on $\QSV_{n}$}
\label{sec:bruhatballot}

The Bruhat order on $S_{n}$ described in Section~\ref{sec:bruhat} restricts to a partial order on the set $\QSV_{n}$.
This order turns out to be very natural, as is described in the paper~\cite{GobetWilliams}, and in this section we recall the description for use in later sections.

Define a partial order $\preceq$ on the set $\NCP_{n}$ of noncrossing partitions as the extension of the covering relation: $\lambda$ is covered by $\mu$ if and only if $\lambda$ is obtained from $\mu$ in one of the following ways:
\begin{enumerate}
\item removing an arc of the form $ \edge{i}{i+1}$ from $\mu$, or

\item replacing any arc $ \edge{i}{k}$ in $\mu$ with two arcs $ \edge{i}{j}$ and $ \edge{j}{k}$ for some $i < j < k$ which do not intersect or share a left or right endpoint with any other arc in $\mu$.

\end{enumerate}

It is difficult to describe the non-covering relations of $\preceq$ on $\NCP_{n}$---and of the Bruhat order on $\QSV_{n}$---in a direct and intuitive manner.  Instead, these relations are best understood through an intermediary object.  A \emph{ballot sequence} of length $2n$ is a sequence $b \in \{\pm1\}^{2n}$ for which each partial sum of $b$ is nonnegative and the final sum is $0$.  A well-known bijection between noncrossing partitions is used in~\cite[Section~5.1]{GobetWilliams}: for $\lambda \in \NCP_{n}$ define a ballot sequence $b^{\lambda} = (b^{\lambda}_{1}, b^{\lambda}_{2}, \ldots, b^{\lambda}_{2n})$  by
\[
b^{\lambda}_{2 k - 1} = \begin{cases} 1 & \text{if $k \notin \lambda^{-}$} \\ -1 & \text{if $k \in \lambda^{-}$} \end{cases}
\qquad
\text{and}
\qquad
b^{\lambda}_{2 k} = \begin{cases} 1 & \text{if $k \in \lambda^{+}$} \\ -1 & \text{if $k \notin \lambda^{+}$} \end{cases}
\]
for each $1 \le k \le n$.

\begin{prop}[{\cite[Theorem 1.1 and Corollary 7.5]{GobetWilliams}}]
\label{prop:QSVorderbijection}
Let $\lambda$ and $\mu$ be noncrossing partitions of size $n$.  The following are equivalent:
\begin{enumerate}[itemsep =1ex]
\item $\lambda \preceq \mu$, 

\item $Q_{\lambda} \le Q_{\mu}$ in the Bruhat order, and

\item for all $1 \le k \le 2n$, $\sum_{i = 1}^{k} b^{\lambda}_{k} \le \sum_{i = 1}^{k} b^{\mu}_{k}$.

\end{enumerate}
\end{prop}

For $n = 3$ the (isomorphic) orders on $\QSV_{n}$, $\NCP_{n}$, and ballot sequences of length $2n$ are shown in Figure~\ref{fig:Hassediagrams}.

\begin{figure}
\begin{center}
\begin{tikzpicture}
\node at (0, 4.5) (13b2) {$321$};
\node at (0, 3) (123) {$312$};
\node at (1.5, 1.5) (1b23) {$132$};
\node at (-1.5, 1.5) (12b3) {$213$};
\node at (0, 0) (1b2b3) {$123$};
\draw[thick] (1b2b3) -- (1b23);
\draw[thick] (1b2b3) -- (12b3);
\draw[thick] (1b23) -- (123);
\draw[thick] (12b3) -- (123);
\draw[thick] (123) -- (13b2);
\end{tikzpicture}
\hspace{0.75cm}
\begin{tikzpicture}
\node at (0, 4.5) (13b2) {\tikz[scale = 0.75, baseline = 0.75*-0.2]{
\foreach \x in {1, ..., 3}{\draw[fill] (\x - 1, 0) node[inner sep = 2pt] (\x) {$\scriptstyle \x$};}
\foreach \i\j in {1/3}{\draw[thick] (\i) to[out = 35, in = 145] (\j);}}};
\node at (0, 3) (123) {\tikz[scale = 0.75, baseline = 0.75*-0.2]{
\foreach \x in {1, ..., 3}{\draw[fill] (\x - 1, 0) node[inner sep = 2pt] (\x) {$\scriptstyle \x$};}
\foreach \i\j in {1/2, 2/3}{\draw[thick] (\i) to[out = 35, in = 145] (\j);}}};
\node at (1.5, 1.5) (1b23) {\tikz[scale = 0.75, baseline = 0.75*-0.2]{
\foreach \x in {1, ..., 3}{\draw[fill] (\x - 1, 0) node[inner sep = 2pt] (\x) {$\scriptstyle \x$};}
\foreach \i\j in {2/3}{\draw[thick] (\i) to[out = 35, in = 145] (\j);}}};
\node at (-1.5, 1.5) (12b3) {\tikz[scale = 0.75, baseline = 0.75*-0.2]{
\foreach \x in {1, ..., 3}{\draw[fill] (\x - 1, 0) node[inner sep = 2pt] (\x) {$\scriptstyle \x$};}
\foreach \i\j in {1/2}{\draw[thick] (\i) to[out = 35, in = 145] (\j);}}};
\node at (0, 0) (1b2b3) {\tikz[scale = 0.75, baseline = 0.75*-0.2]{
\foreach \x in {1, ..., 3}{\draw[fill] (\x - 1, 0) node[inner sep = 2pt] (\x) {$\scriptstyle \x$};}
\foreach \i\j in {}{\draw[thick] (\i) to[out = 35, in = 145] (\j);}}};
\draw[thick] (1b2b3) -- (1b23);
\draw[thick] (1b2b3) -- (12b3);
\draw[thick] (1b23) -- (123);
\draw[thick] (12b3) -- (123);
\draw[thick] (123) -- (13b2);
\end{tikzpicture}
\hspace{0.75cm}
\begin{tikzpicture}
\node at (0, 4.5) (13b2) {$111---$};
\node at (0, 3) (123) {$11-1--$};
\node at (1.5, 1.5) (1b23) {$1-11--$};
\node at (-1.5, 1.5) (12b3) {$11--1-$};
\node at (0, 0) (1b2b3) {$1-1-1-$};
\draw[thick] (1b2b3) -- (1b23);
\draw[thick] (1b2b3) -- (12b3);
\draw[thick] (1b23) -- (123);
\draw[thick] (12b3) -- (123);
\draw[thick] (123) -- (13b2);
\end{tikzpicture}
\end{center}
\caption{The Hasse diagrams of: $\QSV_{3}$ with the Bruhat order; $\NCP_{3}$ with $\preceq$; and ballot sequences (for which each $-1$ is represented as $-$) with the term-wise order on partial sums.}
\label{fig:Hassediagrams}
\end{figure}
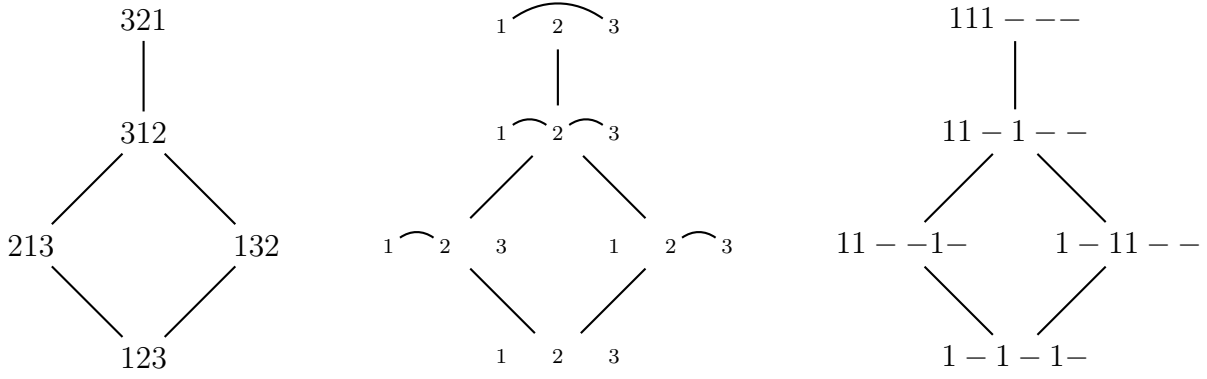

The final result of the section is an easy consequence of the results of~\cite{GobetWilliams}, but is not stated explicitly in this source.  For the sake of completeness, a proof is included.

\begin{cor}
\label{cor:bruhatncphelper}
Let $\lambda$ and $\mu$ be noncrossing partitions of size $n$.  Then $\lambda \preceq \mu$ if and only if 
\[
\big|\lambda^{+} \cap [k-1]\big| - \big|\lambda^{-} \cap [k]\big| \ \le\  \big|\mu^{+} \cap [k-1]\big| - \big|\mu^{-} \cap [k]\big|
\]
and
\[
\big|\lambda^{+} \cap [k]\big| - \big|\lambda^{-} \cap [k]\big|\  \le \  \big|\mu^{+} \cap [k]\big| - \big|\mu^{-} \cap [k]\big| 
\]
for all $1 \le k \le n$.
\end{cor}
\begin{proof}
By Proposition~\ref{prop:QSVorderbijection}, it is sufficient to show that for all $1 \le k \le n$,
\[
\sum_{i = 1}^{2k - 1} b^{\lambda}_{i}  \ = \ 1 + 2\big|\lambda^{+} \cap [k-1]\big| - 2\big|\lambda^{-} \cap [k]\big|
\qquad\text{and}\qquad
\sum_{i = 1}^{2k} b^{\lambda}_{i} \  = \ 2\big|\lambda^{+} \cap [k]\big| - 2\big|\lambda^{-} \cap [k]\big|.
\] 
This will be established inductively: for $k = 1$ the equations can be verified directly, and for $k > 1$, we consider the differences between the $k-1$st expression and the $k$th: 
\[
\Big(1 + 2\big|\lambda^{+} \cap [k-1]\big| - 2\big|\lambda^{-} \cap [k]\big|\Big) - \Big(2\big|\lambda^{+} \cap [k-1]\big| + 2\big|\lambda^{-} \cap [k-1]\big|\Big)
= 1 - 2\big|\lambda^{-} \cap \{k\}\big|,
\]
which is $b^{\lambda}_{2k-1}$, and 
\[
\Big(2\big|\lambda^{+} \cap [k]\big| - 2\big|\lambda^{-} \cap [k]\big|\Big) - \Big(1 + 2\big|\lambda^{+} \cap [k-1]\big| + 2\big|\lambda^{-} \cap [k]\big|\Big) = 2\big|\lambda^{+} \cap \{k\}\big| - 1,
\]
which is $b^{\lambda}_{2k}$.
\end{proof}

\begin{rem}\label{rem:auto}
We make several additional remarks about Proposition~\ref{prop:QSVorderbijection} below.
\begin{enumerate}
\item The results of~\cite{GobetWilliams} describe the Bruhat order of the set $\{\omega_0w\omega_0^{-1}\; |\; w \in \QSV_{n}\}$ rather than $\QSV_{n}$, where $\omega_0\in S_n$ denotes the Bruhat-maximal element $\omega_0\in S_n$. 
In the terminology of Remark~\ref{rem:QSVnoncrossing}, these are the non-crossing partitions associated with the cycle $(12\ldots n)$ instead of $(n\ldots 21)$, but conjugation by $\omega_{0}$ is an automorphism of the Bruhat order so the results are equivalent.  Many of our results have counterparts under this bijection, and we revisit this equivalence throughout the paper.


\item The results of~\cite{GobetWilliams} use Dyck paths rather than ballot sequences for item 3.~in Propostition~\ref{prop:QSVorderbijection}.  
One can translate between the two by interchanging each $1$ in a ballot sequence with an up step in a Dyck path, and likewise each $-1$ with a down step.

\item Another realization of the order in item 3.~of Propostition~\ref{prop:QSVorderbijection} can be found in the dual of the interval between the empty diagram and the staircase in Young's lattice.  This follows from the straightforward order isomorphism between Dyck paths and the aforementioned dual interval~\cite[Solution 6.19 vv]{Stanley}.

\end{enumerate}
\end{rem}

\section{The excedance quotient of the Bruhat order}
\label{sec:excedance}

In this section we describe a novel equivalence relation $\sim$ on $S_{n}$ and show that it induces a quotient of the Bruhat order.  
This equivalence relation is defined in a simple way using the weak excedances of a permutation.  
We have discovered a number of nice properties of the equivalence classes in $S_{n}\big/\!\!\sim$, which we summarize after our initial definition.

%
%
%

Given a permutation $w \in S_{n}$, a \emph{weak excedance} of $w$ is a pair $(i, w_{i})$ for which $i \le w_{i}$.  Disaggregating, we define the \emph{excedance values} $\EV(w)$ and \emph{excedance positions} $\EP(w)$ to be the sets
\begin{align*}
\EV(w) &= \{ w_{i} \;|\; \text{$(i, w_{i})$ is a weak excedance of $w$}\},\;\text{and} \\[0.5em]
\EP(w) &= \{ i  \;|\; \text{$(i, w_{i})$ is a weak excedance of $w$}\}.
\end{align*}
Weak excedances and the sets $\EV(w)$ and $\EP(w)$ are most easily seen using two-line notation for permutations.  For example, marking the non-excedances of a permutation in red,
\[
w = \overset{1}{3} \overset{2}{5} {\color{red} \overset{3}{1}} \overset{4}{4} {\color{red} \overset{5}{2}} \overset{6}{6} {\color{red} \overset{7}{5}} \overset{8}{8},
\qquad
\EP(w) = \{1, 2, 4, 6, 8\},\qquad\text{and}\qquad
\EV(w) = \{3, 4, 5, 6, 8\}.
\]
We define the \emph{excedance relation} $\sim$ on $S_{n}$ by:
\begin{equation}
\label{eq:excednacerel}
v \sim w \qquad\text{if and only if} \qquad \text{$\EV(v) = \EV(w)$ and $\EP(v) = \EP(w)$},
\end{equation}
and say that each equivalence class of $S_{n}\big/\!\!\sim$ is an \emph{excedance class}.

We now summarize the main results of the section.  
Each noncrossing partition $\lambda$ of size $n$ determines an excedance class: recall the sets $\lambda^{+}$ and $\lambda^{-}$ defined in Section~\ref{sec:ncp}, and let
\[
\CCC_{\lambda} = \{ w \in S_{n} \;|\;  \text{$\EV(w) = [n] \setminus \lambda^{+}$ and $\EP(w) = [n] \setminus \lambda^{-}$}  \}.
\]
In Section~\ref{sec:ex1}, we show that this construction is bijective, so that the excedance classes are counted by the Catalan numbers.  For example, the five excedance classes of $S_{3}$ are:
\[
\mathcal{C}_{\begin{tikzpicture}[scale = 0.35, baseline = 0.35*-0.2]
\foreach \x in {1, ..., 3}{\draw[fill] (\x - 1, 0) circle (2pt) node[inner sep = 2pt] (\x) {};}
\foreach \x in {1, ..., 3}{\node[below] at (\x) {$\scriptstyle \x$};}
\foreach \i\j in {1/3}{\draw[thick] (\i) to[out = 35, in = 145] (\j);}
\end{tikzpicture}} = \{ \stackrel{1}{3}  \stackrel{2}{2} \color{red} \stackrel{3}{1} \color{black}, \stackrel{1}{2}\stackrel{2}{3} \color{red} \stackrel{3}{1} \color{black} \}, \qquad
\mathcal{C}_{\begin{tikzpicture}[scale = 0.35, baseline = 0.35*-0.2]
\foreach \x in {1, ..., 3}{\draw[fill] (\x - 1, 0) circle (2pt) node[inner sep = 2pt] (\x) {};}
\foreach \x in {1, ..., 3}{\node[below] at (\x) {$\scriptstyle \x$};}
\foreach \i\j in {1/2, 2/3}{\draw[thick] (\i) to[out = 35, in = 145] (\j);}
\end{tikzpicture}} = \{ \stackrel{1}{3} \color{red} \stackrel{2}{1} \stackrel{3}{2}  \color{black} \}, \qquad
\mathcal{C}_{\begin{tikzpicture}[scale = 0.35, baseline = 0.35*-0.2]
\foreach \x in {1, ..., 3}{\draw[fill] (\x - 1, 0) circle (2pt) node[inner sep = 2pt] (\x) {};}
\foreach \x in {1, ..., 3}{\node[below] at (\x) {$\scriptstyle \x$};}
\foreach \i\j in {1/2}{\draw[thick] (\i) to[out = 35, in = 145] (\j);}
\end{tikzpicture}} = \{\stackrel{1}{2}  \color{red} \stackrel{2}{1} \color{black}  \stackrel{3}{3}\}, 
\]
\[
\mathcal{C}_{\begin{tikzpicture}[scale = 0.35, baseline = 0.35*-0.2]
\foreach \x in {1, ..., 3}{\draw[fill] (\x - 1, 0) circle (2pt) node[inner sep = 2pt] (\x) {};}
\foreach \x in {1, ..., 3}{\node[below] at (\x) {$\scriptstyle \x$};}
\foreach \i\j in {2/3}{\draw[thick] (\i) to[out = 35, in = 145] (\j);}
\end{tikzpicture}} = \{\stackrel{1}{1} \stackrel{2}{3}  \color{red} \stackrel{3}{2} \color{black} \}, \qquad\text{and}\qquad
\mathcal{C}_{\begin{tikzpicture}[scale = 0.35, baseline = 0.35*-0.2]
\foreach \x in {1, ..., 3}{\draw[fill] (\x - 1, 0) circle (2pt) node[inner sep = 2pt] (\x) {};}
\foreach \x in {1, ..., 3}{\node[below] at (\x) {$\scriptstyle \x$};}
\end{tikzpicture}} = \{\stackrel{1}{1} \stackrel{2}{2} \stackrel{3}{3} \}.
\]

The Bruhat order induces a relation on $S_{n}/\!\!\sim$, and in Section~\ref{sec:ex3} we show that this relation is a partial order by way of the following result.  Recall the order $\preceq$ from Section~\ref{sec:bruhatballot}.

\begin{thm}
\label{thm:excedancequotient}
Writing $\le$ for the relation on excedance classes $S_{n}\big/\!\! \sim$ induced by the Bruhat order, $\CCC_{\lambda} \le \CCC_{\mu}$ if and only if $\lambda \preceq \mu$.
\end{thm}

Two key intermediate results in the proof of Theorem~\ref{thm:excedancequotient} appear in Section~\ref{sec:ex2}: we show that each excedance class $\CCC_{\lambda}$ contains unique Bruhat-minimal and Bruhat-maximal elements, and moreover that these are respectively a $321$-avoiding permutation and the element $Q_{\lambda} \in \QSV_{n}$.  Combined with Theorem~\ref{thm:excedancequotient}, this implies the following corollary.  

\begin{cor}\label{cor:interval}
Each excedance class $\CCC_{\lambda}$ is an interval in the Bruhat order, with upper bound $Q_{\lambda}\in \QSV_n$ and lower bound given by a $321$-avoiding permutation
\end{cor}


\begin{rem}
\label{rem:altECs}
Two alternative notions of ``excedance classes'' are worth mentioning.
\begin{enumerate}
\item Our definition of excedance class is one of two obvious choices: one can instead consider the equivalence classes determined by sets $\EV'(w)$ and $\EP'(w)$ of values $i$ and positions $w_{i}$ for which $i < w_{i}$, known as strict excedances.  
This choice is equally viable and would lead to results equivalent to the ones presented here.  This equivalence is directly  related to the automorphism described in Remark~\ref{rem:auto} (1).

\item The excedance relation refines an equivalence relation defined in~\cite{ES}, for which permutations are related if they share the same excedance position set.  
The authors of~\cite{ES} prove several enumerative and statistical results about their ``excedance classes,'' and it is worth considering how these results extend to our notion of excedance class.  
%

\end{enumerate}
\end{rem}

\subsection{Excedance classes and noncrossing partitions}
\label{sec:ex1}

This section will establish some basic results which relate excedance classes to noncrossing partitions.  

\begin{prop}
For $n \ge 0$, the map
\[
\begin{array}{ccc}
\NCP_{n} & \longrightarrow & S_{n}\big/\!\!\sim \\
\lambda & \longmapsto & \CCC_{\lambda}
\end{array}
\]
is a bijection.
\end{prop}
\begin{proof}
We will first show that every excedance class has the form $\CCC_{\lambda}$ for some $\lambda \in \NCP_{n}$.  
By Lemma~\ref{lem:noncrossingpartitionproperty}, this is equivalent to showing that for any $w \in S_{n}$, the following criterion holds:
\[
\Big|\big\{i \in [k-1] \;|\; i \notin \EV(w)\big\}\Big| \ \ge\  \Big|\big\{ j \in [k] \;|\; j \notin \EP(w)\big\}\Big| \qquad\text{for all $1 \le k \le n$}.
\]
We will establish the above inequality directly.  Fix $1 \le k \le n$ and suppose that $j \in [k]$ is not an excedance position, so that that $w_{j} < j$.  By assumption, $w_{j} < k$ and $w_{j}$ is not an excedance value.  Thus, 
\[
\big\{i \;|\; \text{$i \in [k-1]$ and $i \notin \EV(w)$} \big\}\  \supseteq\  \big\{w_{j} \;|\; \text{$j \in [k]$ and $j \notin \EP(w)$}\big\},
\]
giving the claim.  Now, we must show that each $\CCC_{\lambda}$ is nonempty.  This is established by Lemma~\ref{lem:QSVexcedance} below, which completes the proof.
\end{proof}

\begin{lem}
\label{lem:QSVexcedance}
Let $\lambda$ be a noncrossing partition of size $n$.  Then $Q_{\lambda} \in \CCC_{\lambda}$, so that
\[
\EV(Q_{\lambda}) = [n] \setminus \lambda^{+}
\qquad\text{and}\qquad
\EP(Q_{\lambda}) = [n] \setminus \lambda^{-}.
\]
\end{lem}
\begin{proof}
By definition, $Q_{\lambda}(j) < j$ if and only if the arc $ \edge{Q_{\lambda}(j)}{j}$ appears in $\lambda$, in which case $j \in \lambda^{-}$ and $Q_{\lambda}(j) \in \lambda^{+}$; this establishes the claim.
\end{proof}

The final result relates certain properties of the elements of $\CCC_{\lambda}$ to the partial sums of the ballot sequence $b^{\lambda}$ defined in Section~\ref{sec:bruhatballot}, by way of Corollary~\ref{cor:bruhatncphelper}.  This will be key to a number of arguments in subsequent sections.

\begin{lem}
\label{lem:tableautobruhat}
Let $\lambda$ be a noncrossing partition of size $n \ge 0$ and take $w \in \CCC_{\lambda}$.  For all $1 \le k \le n$, 
\[
\big|\{ w_{i} \;|\; \text{$n \ge i > k$ and $w_{i} < k$}\}\big|
\ = \ \big|\lambda^{+} \cap [k-1]\big| - \big|\lambda^{-} \cap [k]\big|
\]
and 
\[
\big|\{ w_{i} \;|\; \text{$1 \le i \le k$ and $w_{i} > k$}\}\big|
\ =\  \big|\lambda^{+} \cap [k]\big| - \big|\lambda^{-} \cap [k]\big|.
\]
\end{lem}
\begin{proof}
To show the first equation, note that by definition
\[
\big|\lambda^{+} \cap [k-1]\big| - \big|\lambda^{-} \cap [k]\big| \ = \ \Big|\big\{ w_{i} \in [k-1] \;|\; w_{i} \notin \EV(w)\big\}\Big| - \Big|\big\{i \in [k] \;|\; i \notin \EP(w)\big\}\Big|.
\]
The permutation $w$ gives a bijection between the $i \in [k]$ with $i \notin \EP(w)$ and the $w_{i} \in [k-1]$ with both $w_{i} \notin \EV(w)$ and $i \in [k]$.  
Therefore, the right hand side above counts the $w_{i} \in [k-1]$ for which $w_{i} \notin \EV(w)$ and $i \notin [k]$, which is exactly $\{ w_{i} \;|\; \text{$n \ge i > k$ and $w_{i} < k$}\}$.

The second equation follows from a similar but somewhat more complicated argument.  We begin by manipulating the right side into a more suitable form:
\begin{align*}
\big|\lambda^{+} \cap [k]\big| - \big|\lambda^{-} \cap [k] \big| \ &= \  \Big(k - \big|\lambda^{-} \cap [k] \big|\Big) - \Big(k - \big|\lambda^{+} \cap [k]\big|\Big) \\
&=\  \big|\EP(w) \cap [k]\big| - \big|\EV(w) \cap [k]\big|.
\end{align*}
For each $w_{i} \in \EV(w) \cap [k]$ we have $i \le w_{i}$, and therefore $i \in \EP(w) \cap [k]$.  Thus, $w^{-1}$ gives a bijection between $\EV(w) \cap [k]$ and the subset of $i \in \EP(w) \cap [k]$ for which $w_{i} \le k$.  Therefore, the equation above counts the positions $i \in \EP(w) \cap [k]$ for which $w_{i} \notin [k]$; this set is equinumerous to $\{ w_{i} \;|\; \text{$1 \le i \le k$ and $w_{i} > k$}\}$.
\end{proof}

\subsection{Minimal and maximal elements}
\label{sec:ex2}

This section will show that each excedance class contains a unique Bruhat minimum and maximum.  
This is a key intermediate step to showing that excedance classes are Bruhat intervals with a well-defined quotient order.  
We will begin with the maximal elements, which are the elements of $\QSV_{n}$, while the minimal ($321$-avoiding) elements are discussed after the proof of Proposition~\ref{prop:QSVinterval} below.

\begin{prop}
\label{prop:QSVinterval}
For all noncrossing partitions $\lambda$, $Q_{\lambda}$ is the Bruhat maximum element of $\CCC_{\lambda}$.  
\end{prop}

\begin{proof}
Lemma~\ref{lem:QSVexcedance} shows that $Q_{\lambda} \in \CCC_{\lambda}$, so we only need to show that $Q_{\lambda}$ is an upper bound for $\CCC_{\lambda}$; to this end, fix $w \in \CCC_{\lambda}$.  Using the tableau criterion (Proposition~\ref{TableauCriterion}), it is sufficient to show that each entry of $\TTT_{k}(w)$ is bounded above by the corresponding entry of $\TTT_{k}(Q_{\lambda})$ for each $1 \le k \le n$.

The argument consists of two distinct parts, first conducting an element-by-element comparison of the entries of $\TTT_{k}(w)$ and $\TTT_{k}(Q_{\lambda})$ which are strictly greater than $k$, and then doing the same for the elements which are at most $k$.  The validity of this approach relies on the fact that these collections of entries have the same cardinality for $w$ and $Q_{\lambda}$: by Lemma~\ref{lem:tableautobruhat}, 
\begin{align*}
\big|\{ w_{i} \;|\; \text{$1 \le i \le k$ and $w_{i} > k$}\}\big|
&= \big|\lambda^{+} \cap [k]| - |\lambda^{-} \cap [k] \big| \\
&= \big|\{ Q_{\lambda}(i) \;|\; \text{$1 \le i \le k$ and $w_{i} > k$}\}\big|\,.
\end{align*}

We begin with the first part.  Enumerate the entries of $\TTT_{k}(w)$ and $\TTT_{k}(Q_{\lambda})$ which are greater than $k$ in increasing order as $x_{1} < x_{2} < \cdots <  x_{r}$ and $q_{1} < q_{2} < \ldots < q_{r}$ respectively.  We aim to show that $x_{i} \le q_{i}$ for each $1 \le i \le r$.  
Fixing one such $i$, Lemma~\ref{lem:tableautobruhat} gives that
\[
\big|\{ w_{t} \;|\; \text{$1 \le t \le q_{i}$ and $w_{t} > q_{i}$} \}\big| = \big|\lambda^{+} \cap [q_{i}]| - |\lambda^{-} \cap [q_{i}]\big|.
\]
It is sufficient to show that the above expressions are equal to $r - i$, as this implies that each of $x_{1}, x_{2}, \ldots, x_{i}$ is not contained in the set $\{ w_{t} \;|\; \text{$1 \le t \le q_{i}$ and $w_{t} > q_{i}$} \}$, and therefore that $x_{i} \le q_{i}$.

Let $C$ denote the connected component of $\lambda$ containing the value $q_{i}$.  Since $q_i\in \TTT_{k}(Q_{\lambda})$ and $q_i>k$, we have that $t_i=Q_{\lambda}^{-1}(q_{i})\le k <q_i$.
From the definition of $Q_{\lambda}$, $q_{i}$ must be the maximal element of $C$,  and $t_i$ the minimal element.  
Combinatorially, the difference $\big|\lambda^{+} \cap [q_{i}]\big| - \big|\lambda^{-} \cap [q_{i}]\big|$ counts the number of arcs in $\lambda$ with left endpoint in $[q_{i}]$ and right endpoint in $[n] \setminus [q_{i}]$, and as $q_{i} \in \EV(w)$, this is the number of arcs in $\lambda$ which are above $q_{i}$.  
Every arc in $\lambda$ which lies above $q_{i}$ must have all elements of $C$ between its left and right endpoints, so that the left endpoint of any such an arc is contained in $[k]$ and the right endpoint is greater than $q_{i}$.  
Thus, each such arc belongs to the connected component of one of the elements $q_{i+1}, \ldots, q_{r}$, and there are precisely $r - i$ such connected components.

For the second part of the argument, we aim to show that each entry of $\TTT_{k}(w)$ which is at most $k$ is less than or equal to the analogous entry of $\TTT_{k}(Q_{\lambda})$, and we establish this in an indirect manner described below.  
Writing $s = \big|\lambda^{+} \cap [k-1]\big| - \big|\lambda^{-} \cap [k]\big|$, Lemma~\ref{lem:tableautobruhat} states that there are exactly $s$ elements of $[k]$ which do not appear in $\TTT_{k}(w)$, and likewise for $\TTT_{k}(Q_{\lambda})$.  
Respectively enumerate these elements in increasing order as $y_{1} < y_{2} < \cdots < y_{s}$ and $p_{1} < p_{2} < \cdots < p_{s}$.  
We will show that $p_{i} \le y_{i}$ for each $1 \le i \le s$, as this implies the opposite comparison for the remaining elements of $[k] \setminus \{p_{1}, p_{2}, \ldots, p_{s}\}$ and $[k] \setminus \{y_{1}, y_{2}, \ldots, y_{s}\}$ as desired.  

Fixing $1 \le i \le s$, Lemma~\ref{lem:tableautobruhat} gives that
\[
\big|\{ w_{t} \;|\; \text{$n \ge t > k$ and $w_{t} < p_{i}$} \}\big| = \big|\lambda^{+} \cap [p_{i} - 1]\big| - \big|\lambda^{-} \cap [p_{i}]\big|.
\]
It is therefore sufficient to show that the above quantity is $i-1$, so that $p_{i}$ is bounded above by each of $y_{i}, y_{i+1}, \ldots, y_{s}$, and in particular $p_{i} \le y_{i}$.

Combinatorially, the difference $\big|\lambda^{+} \cap [p_{i} - 1]\big| - \big|\lambda^{-} \cap [p_{i}]\big|$ counts the number of arcs in $\lambda$ with a left endpoint $[p_{i} -1]$ and a right endpoint in $[n] \setminus [p_{i}]$, or equivalently, the arcs above $p_{i}$ in $\lambda$.  
Writing $C$ for the connected component containing $p_{i}$, each such arc must contain $C$ between its left and right endpoints, so the left endpoint is less that $p_{i}$ and the right endpoint lies somewhere in $[n] \setminus [k]$.  
Thus, each such arc must belong to the connected component of one of $p_{1}, p_{2}, \ldots, p_{i-1}$, and there are $i-1$ such connected components.
\end{proof}


We now turn to the minimal element of each excedance class.  For a noncrossing partition $\lambda$ of size $n$, enumerate the sets $\lambda^{+}$, $\lambda^{-}$, $[n] \setminus \lambda^{+}$, and  $[n] \setminus \lambda^{-}$ in increasing order as
\[
\lambda^{+} = \{a_{1} < a_{2} < \cdots < a_{s}\},
\qquad
\lambda^{-} = \{b_{1} < b_{2} < \cdots < b_{s}\},
\]
\[
[n] \setminus \lambda^{+} = \{x_{1} < x_{2} < \cdots < x_{n-s}\},
\qquad\text{and}\qquad
[n] \setminus \lambda^{-} = \{y_{1} < y_{2} < \cdots < y_{n-s}\}.
\]
Let $T_{\lambda} \in S_{n}$ be the permutation with
\[
T_{\lambda}(i) = \begin{cases} a_{r} & \text{if $i \in \lambda^{-}$ and $i = b_{r}$} \\ x_{r} & \text{if $i \notin \lambda^{-}$ and $i = y_{r}$.}  \end{cases}
\]
Thus, the two-line notation for $T_{\lambda}$ can be obtained by placing the elements of $\lambda^{+}$ in increasing left-to-right order below the elements of $\lambda^{-}$, and placing the elements of $[n] \setminus \lambda^{+}$ below the elements of $[n] \setminus \lambda^{-}$ in the same manner.  For example, with $n = 8$ and 
\[
\lambda = \begin{tikzpicture}[scale = 0.75, baseline = 0.75*-0.2]
\foreach \x in {1, ..., 8}{\draw[fill] (\x - 1, 0) node[inner sep = 2pt] (\x) {$\scriptstyle \x$};}
\foreach \i\j in {1/5, 2/3, 5/7}{\draw[thick] (\i) to[out = 35, in = 145] (\j);}
\end{tikzpicture}
\]
we have $\lambda^{+} = \{1, 2, 5\}$ and $\lambda^{-} = \{3, 5, 7\}$, $[8] \setminus \lambda^{+} = \{3, 4, 6, 7, 8\}$, and $[8] \setminus \lambda^{-} = \{1, 2, 4, 6, 8\}$, and consequently
\[
T_{\lambda} = \overset{1}{3} \overset{2}{4} {\color{red} \overset{3}{1}} \overset{4}{6} {\color{red} \overset{5}{2}} \overset{6}{7} {\color{red} \overset{7}{5}} \overset{8}{8},
\]
where non-excedances are marked in red, as at the beginning of Section~\ref{sec:excedance}.

\begin{prop}
\label{prop:321avoid}
For all noncrossing partitions $\lambda$, $T_{\lambda} \in \CCC_{\lambda}$, and $T_{\lambda}$ is the Bruhat-minimum element of $\CCC_{\lambda}$.
\end{prop}
\begin{proof}
To see that $T_{\lambda} \in \CCC_{\lambda}$, recall the elements $a_{i}, b_{i}, x_{i}$, and $y_{i}$ defined above for $\lambda$. For $1 \le r \le s$, our enumeration of $\lambda^{+}$ and $\lambda^{-}$ ensure that
\[
a_{r} = \min\Big\{ k  \;|\;  r \ge \big|[k] \cap \lambda^{+}\big| \Big\}
\qquad\text{and}\qquad
b_{r} = \min\Big\{ k  \;|\;  r \ge \big|[k] \cap \lambda^{-}\big| \Big\}.
\]
By Lemma~\ref{lem:noncrossingpartitionproperty}, it is always the case that $\big|[k-1] \cap \lambda^{+}\big| \ge \big|[k] \cap \lambda^{-}\big|$, and so $a_{r} < b_{r}$.  Thus the pair $(b_{r}, a_{r}) = (b_{r}, T_{\lambda}(b_{r}))$ is not a weak excedance of $T_{\lambda}$.  
A similar argument shows that every pair $(y_{r}, x_{r})$, $1 \le r \le n-s$ is a weak excedance of $T_{\lambda}$, so we conclude that
\[
\lambda^{+} = [n] \setminus \EV(T_{\lambda})
\qquad\text{and}\qquad
\lambda^{-} =  [n] \setminus \EP(T_{\lambda}).
\]

To see that $T_{\lambda} \le w$ for all $w \in \CCC_{\lambda}$, recall the tableau criterion (Proposition~\ref{TableauCriterion}).  
For $1 \le k \le n$, the row $\TTT_{k}(T_{\lambda})$ will consist of the $\big|[k] \cap \lambda^{-}\big|$ smallest elements of $\EV(w)$ along with the $k - |[k] \cap \lambda^{-}|$ smallest elements of $[n] \setminus \EV(w)$.  
For $w$,  the row $\TTT_{k}(w)$ will also consist of $|[k] \cap \lambda^{-}|$ elements of $\EV(w)$ and $k - |[k] \cap \lambda^{-}|$ elements of $[n] \setminus \EV(w)$, but these elements need not be the smallest ones.  
Thus, by assumption of minimality, each entry of $\TTT_{k}(T_{\lambda})$ is bounded above by the corresponding entry of $\TTT_{k}(w)$.
\end{proof}

Finally, recall that a permutation $w \in S_{n}$ is \emph{$321$-avoiding} if there do not exist indices $i < j < k$ for which $w_{i} > w_{j} > w_{k}$.  The number of $321$-avoiding permutations is known to be the $n$th Catalan number, so the following result establishes that each excedance class contains a unique $321$-avoiding permutation.

\begin{prop}
For all noncrossing partitions $\lambda$, the permutation $T_{\lambda}$ is $321$-avoiding.
\end{prop}
\begin{proof}
Let $i, j, k\in [n]$ and assume without loss of generality that $i < j < k$.  Since any element of $[n]$ must be contained in either $\lambda^{+}$ or its complement, we must have two elements of $\{i, j, k\}$ which belong to one of $\lambda^{+}$ or $[n] \setminus \lambda^{+}$.  Since $T_{\lambda}$ restricts to an order-preserving bijection from $\lambda^{+}$ to $\lambda^{-}$, and from $[n] \setminus \lambda^{+}$ to $[n] \setminus \lambda^{-}$, this implies that no $321$-pattern can exist in $T_{\lambda}$, completing the proof.
\end{proof}

\subsection{Comparing excedance classes}
\label{sec:ex3}

In this section we give a proof of Theorem~\ref{thm:excedancequotient}.  This proof follows a final intermediate result; recall the order $\preceq$ on noncrossing partitions defined in Section~\ref{sec:bruhatballot}

\begin{prop}
\label{prop:Qdominates}
Let $\mu$ be a noncrossing partition of size $n$.  Then
\[
\{w \in S_{n} \;|\; w \le Q_{\mu}\} = \bigsqcup_{\lambda \preceq \mu} \CCC_{\lambda}.
\]
\end{prop}
\begin{proof}
To begin, assume that $\lambda \preceq \mu$, so that by Proposition~\ref{prop:QSVorderbijection}, $Q_{\lambda} \le Q_{\mu}$.  Proposition~\ref{prop:QSVinterval} states that $w \le Q_{\lambda}$ for any $w \in C_{\lambda}$, so by transitivity we have
\[
\{w \in S_{n} \;|\; w \le Q_{\mu}\} \supseteq \bigsqcup_{\lambda \preceq \mu} \CCC_{\lambda}.
\]

To see the opposite containment, suppose that $w$ is a permutation with $w \le Q_{\mu}$ and let $\lambda$ be the unique noncrossing partition for which $w \in \CCC_{\lambda}$.  
By Corollary~\ref{cor:bruhatncphelper}, the statement $\lambda \preceq \mu$ is equivalent to the inequalities
\[
\big|\lambda^{+} \cap [k-1]\big| - \big|\lambda^{-} \cap [k]\big| \  \le\   \big|\mu^{+} \cap [k-1]\big| - |\mu^{-} \cap [k]\big|
\]
and
\[
\big|\lambda^{+} \cap [k]\big| - \big|\lambda^{-} \cap [k]\big|\  \le \  \big|\mu^{+} \cap [k]\big| - \big|\mu^{-} \cap [k]\big|
\]
for all $1 \le k \le n$.  
We will establish this equivalent formulation of our claim using the characterization of each side given in Lemma~\ref{lem:tableautobruhat}.  
From the assumption that $w \le Q_{\mu}$, the tableau criterion (Proposition~\ref{TableauCriterion}) states that the tableau $\TTT(w)$ is entry-wise less than or equal to $\TTT(Q_{\mu})$.  
Thus, for each $1 \le k \le n$, 
\[
\big|\{w_{i} \;|\; \text{$1 \le i \le k$ and $w_{i} > k$} \}\big|
\ \le\ 
\big|\{Q_{\mu}(i) \;|\; \text{$1 \le i \le k$ and $Q_{\mu}(i) > k$} \}\big|,
\]
since each entry of $\TTT_{k}(w)$ is bounded above by the corresponding entry of $\TTT_{k}(Q_{\mu})$, and likewise
\[
\big|\{ w_{i} \;|\; \text{$n \ge i > k$ and $w_{i} < k$} \}\big| 
\ \le\ 
\big|\{ Q_{\mu}(i) \;|\; \text{$n \ge i > k$ and $Q_{\mu}(i) < k$} \}\big|,
\]
since each entry of $\TTT_{k}(Q_{\mu})$ is bounded below by the corresponding entry of $\TTT_{k}(w)$.  This completes the proof.
\end{proof}

We now prove Theorem~\ref{thm:excedancequotient}.  Recall the elements $T_{\lambda}$ defined in Section~\ref{sec:ex2}.

\begin{proof}[Proof of Theorem~\ref{thm:excedancequotient}]
The relation $\le$ on excedance classes is defined by:
\[
\CCC_{\lambda} \le \CCC_{\mu}
\qquad\text{if and only if}\qquad
\text{$v \le w$ for some $v \in \CCC_{\lambda}$ and $w \in \CCC_{\mu}$}.
\]
As Lemma~\ref{lem:QSVexcedance} states that $Q_{\lambda} \in \CCC_{\lambda}$ for each noncrossing partition $\lambda$, Proposition~\ref{prop:QSVorderbijection} implies that $\CCC_{\lambda} \le \CCC_{\mu}$ whenever $\lambda \preceq \mu$.  It is therefore sufficient to show the converse: $\CCC_{\lambda} \le \CCC_{\mu}$ only if $\lambda \preceq \mu$.  To this end, suppose that $v \le w$ for some $v \in \CCC_{\lambda}$ and $w \in \CCC_{\mu}$.  By Proposition~\ref{prop:QSVinterval}, $w \le Q_{\mu}$, so the transitivity of the Bruhat order on $S_{n}$ implies that $v \le Q_{\mu}$.   Proposition~\ref{prop:Qdominates} now implies that $\lambda \preceq \mu$, completing the proof.
\end{proof}

\section{Bases for the Temperley--Lieb Algebra $\TL_{n}(2)$}
\label{sec:TLbasis}

The Temperley--Lieb algebra $\TL_{n}(2)$ is the $\CC$-algebra generated by elements $e_{1}, \ldots, e_{n-1}$ subject to the following relations for each $1 \le i, j \le n$
\[
\begin{array}{rll}
e_{i}^{2} &= 2 e_{i} \\
e_{i}e_{j} &= e_{j}e_{i} & \text{if $|i - j| > 1$} \\
e_{i} e_{j} e_{i} &= e_{i} & \text{if $|i - j| = 1$}.
\end{array}
\]
There is a surjective algebra morphism from the symmetric group algebra $\CC S_{n}$ to $\TL_{n}(2)$ given by 
\[
\begin{array}{rcl}
\phi: \CC S_{n} & \longrightarrow & \TL_{n}(2) \\
s_{i} & \longmapsto & 1 - e_{i}.
\end{array}
\]
In particular $\TL_{n}(2)\cong S_n\big/\!\ker(\phi)$.

It is well-known that the images of all $321$-avoiding permutations under $\phi$ forms a basis for $\TL_{n}(2)$.  
Another basis, due to Zinno~\cite{Zinno} can be obtained via the map $\phi$ using the combinatorics of noncrossing partitions described in Remark~\ref{rem:QSVnoncrossing}.  This basis is precisely the set $\{\phi(\omega_0w\omega_0) \;|\; w \in \QSV_{n}\}$.  Since the kernel
\[
\ker(\phi) = \langle (13) - (123) - (132) + (12) + (23) - e \rangle
\]
is invariant under the $\CC$-linear map $w \mapsto \omega_0w\omega_0$ (as in Remark~\ref{rem:auto}), the following result is an immediate consequence of~\cite[Theorem 2]{Zinno}.

\begin{thm}
\label{thm:TLbasis}
For all $n \ge 0$, the set $\phi(\QSV_{n})$ is a basis for $\TL_{n}(2)$.
\end{thm}

\begin{rem}
Theorem~\ref{thm:TLbasis} also follows from the work of Gobet and Williams in~\cite{GobetWilliams}: their results imply a  generalization of Zinno's in which each Coxeter element of $S_{n}$ determines a basis of $\TL_{n}(2)$ under $\phi$ (see also Remark~\ref{rem:QSVnoncrossing}).  
However,~\cite{GobetWilliams} does not explicitly state this result about bases, so the theorem is more easily deduced from Zinno's earlier work.
\end{rem}

\subsection{More Bases for the Temperley--Lieb Algebra $\TL_{n}(2)$}

In our investigation of excedance classes we found an application of their structure the problem of computing sets of permutations which give bases of $\TL_{n}(2)$ under $\phi$.  
We include it here with its proof as it is a nice result of our current investigation.

\begin{thm}
\label{thm:TLbases}
Let $n \ge 0$ and for each noncrossing partition $\lambda$ of size $n$, fix an element $w_{\lambda} \in \CCC_{\lambda}$. Then the set $\{\phi(w_{\lambda}) \;|\; \text{noncrossing partitions $\lambda$}\}$ is a basis of $\TL_{n}(2)$.
\end{thm}

A proof of Theorem~\ref{thm:TLbases} is given in Section~\ref{sec:basistheoremproof}.  Here, we discuss its implications: taking $w_{\lambda} = Q_{\lambda}$ in the theorem gives yet another proof of Theorem~\ref{thm:TLbasis}, confirming the results of~\cite{GobetWilliams} and~\cite{Zinno} discussed above.  In general, however, many bases obtained via Theorem~\ref{thm:TLbases} are novel.  The smallest novel example can be found with $n = 4$: the set 
\[
\{
\stackrel{1}{4} \stackrel{2}{3} \color{red} \stackrel{3}{1} \stackrel{4}{2} \color{black}, 
\stackrel{1}{4} \stackrel{2}{2} \stackrel{3}{3}  \color{red} \stackrel{4}{1} \color{black}, 
\stackrel{1}{4} \stackrel{2}{2}\color{red}  \stackrel{3}{1} \stackrel{4}{3} \color{black}, 
\stackrel{1}{3} \color{red}  \stackrel{2}{1} \color{black} \stackrel{3}{4} \color{red} \stackrel{4}{2}\color{black} , 
\stackrel{1}{1} \stackrel{2}{4} \stackrel{3}{3} \color{red} \stackrel{4}{2} \color{black}, 
\stackrel{1}{4} \color{red} \stackrel{2}{1} \stackrel{3}{2} \stackrel{4}{3} \color{black}, 
\stackrel{1}{3} \stackrel{2}{2} \color{red} \stackrel{3}{1} \color{black} \stackrel{4}{4}, 
\stackrel{1}{3} \color{red} \stackrel{2}{1} \stackrel{3}{2} \color{black} \stackrel{4}{4}, 
\stackrel{1}{2} \color{red} \stackrel{2}{1} \color{black} \stackrel{3}{4} \color{red} \stackrel{4}{3} \color{black}, 
\stackrel{1}{1} \stackrel{2}{3} \color{red} \stackrel{3}{2} \stackrel{4}{3} \color{black}, 
\stackrel{1}{2}  \color{red} \stackrel{2}{1} \color{black} \stackrel{3}{3} \stackrel{4}{4}, 
\stackrel{1}{1} \stackrel{2}{3} \color{red} \stackrel{3}{2} \color{black} \stackrel{4}{4}, 
\stackrel{1}{1} \stackrel{2}{2} \stackrel{3}{4}  \color{red} \stackrel{4}{3} \color{black},
\stackrel{1}{1} \stackrel{2}{2} \stackrel{3}{3} \stackrel{4}{4}
\}
\]
meets the criteria of Theorem~\ref{thm:TLbases}, and accordingly maps to a basis of $\TL_{n}(2)$ under $\phi$. 
This set is neither $\QSV_{4}$ nor the set of $321$-avoiding permutations ($4312 \notin \QSV_{4}$ and is not $321$-avoiding).
Moreover, the set above is not described in~\cite{GobetWilliams, Zinno}: each subset of $S_{4}$ in these sources which is not $\QSV_{4}$ contains more than one element from certain excedance classes and none from others.  
 
\begin{rem}
Like other results in this paper, Theorem~\ref{thm:TLbases} implies a similar statement involving the alternate excedance classes in Remark~\ref{rem:altECs} (1).  This statement should give even more bases of $\TL_{n}(2)$, and is obtained by conjugating all permutations in the theorem by $\omega_0$.
\end{rem}

\subsection{A presentation of the Temperley--Lieb algebra}
\label{sec:TLpres}

In this section we describe the Temperley--Lieb algebra $\TL_{n}(2)$ as the quotient of the symmetric group algebra $\CC S_{n}$ by the kernel of the map $\phi$.  
To begin, recall that a $321$-pattern in a permutation $w \in S_{n}$ is a triple $i < j < k$ for which $w_{i} > w_{j} > w_{k}$.  Given such a pattern, we can write
\[
w = \mathbf{a}w_{i}\mathbf{b}w_{j}\mathbf{c}w_{k}\mathbf{d},
\]
where $\mathbf{a}$, $\mathbf{b}$, $\mathbf{c}$, and $\mathbf{d}$ are (possibly empty) subwords of $w$ in one-line notation.  The quotient $\CC S_{n} / \ker(\phi)$ is then defined by the relations
\begin{multline}
\label{eq:321relation}
w \equiv \mathbf{a}w_{j}\mathbf{b}w_{i}\mathbf{c}w_{k}\mathbf{d} + \mathbf{a}w_{i}\mathbf{b}w_{k}\mathbf{c}w_{j}\mathbf{d} - \mathbf{a}w_{j}\mathbf{b}w_{k}\mathbf{c}w_{i}\mathbf{d} \\
- \mathbf{a}w_{k}\mathbf{b}w_{i}\mathbf{c}w_{j}\mathbf{d} + \mathbf{a}w_{k}\mathbf{b}w_{j}\mathbf{c}w_{i}\mathbf{d}  \qquad \text{(mod $\ker(\phi)$)}
\end{multline}
for each $321$-pattern in each permutation $w \in S_{n}$.

\begin{lem}
\label{lem:321reducebruhat}
Let $w \in S_{n}$ be a permutation with a $321$-pattern in positions $i < j < k$.  Then $w > w'$ for each permutation $w'$ in the set
\[
\{
\mathbf{a}w_{j}\mathbf{b}w_{i}\mathbf{c}w_{k}\mathbf{d},  \,
\mathbf{a}w_{i}\mathbf{b}w_{k}\mathbf{c}w_{j}\mathbf{d}, \,
\mathbf{a}w_{j}\mathbf{b}w_{k}\mathbf{c}w_{i}\mathbf{d}, \,
\mathbf{a}w_{k}\mathbf{b}w_{i}\mathbf{c}w_{j}\mathbf{d}, \,
\mathbf{a}w_{k}\mathbf{b}w_{j}\mathbf{c}w_{i}\mathbf{d}
\}
\]
\end{lem}

The lemma is well known, but we sketch a proof for the sake of completeness.

\begin{proof}
We consider the case $w'=\mathbf{a}w_{j}\mathbf{b}w_{i}\mathbf{c}w_{k}\mathbf{d}$, with all other cases following from a similar argument.  
Direct computation shows that $w(w')^{-1}= (w_j\,w_i)$, which is a transposition. 
It is then straightforward to verify that every inversion of $w'$ is an inversion on $w$, and moreover that $w$ has at least one inversion that $w'$ does not: $i < j$ and $w_{i} > w_{j}$.  We therefore conclude that $\ell(w')<\ell(w)$, so that $w' < w$.  
\end{proof}

\subsection{Proof of Theorem~\ref{thm:TLbases}}
\label{sec:basistheoremproof}

This section proves Theorem~\ref{thm:TLbases}, which follows from the triangularity established by the next result.  
Recall the $321$-avoiding permutations $T_{\lambda}$ defined in Section~\ref{sec:ex2} for each noncrossing partition $\lambda \in \NCP_{n}$.  
Further note that the set of all $321$-avoiding permutations in $S_{n}$ form a basis of $\TL_{n}(2)$ under $\phi$; this can be deduced (for example from Equation~\eqref{eq:321relation} and Lemma~\ref{lem:321reducebruhat}.  

\begin{prop}
\label{prop:TLbases}
Let $\lambda$ be a noncrossing partition of size $n$.  For each $w \in \CCC_{\lambda}$, we have
\[
w \equiv T_{\lambda} + \sum_{\mu \prec \lambda} a_{\mu}^{w} T_{\mu} \qquad \text{(mod $\ker(\phi)$)}
\]
for some coefficients $a_{\mu}^{w} \in \ZZ$.
\end{prop}

A proof of Proposition~\ref{prop:TLbases} follows the next lemma.

\begin{lem}
\label{lem:321excedance}
Suppose that $w \in S_{n}$ has a $321$-pattern.  Then $w$ has a $321$-pattern $i < j < k$ with $i \in \EP(w)$ and $k \notin \EP(w)$.  Moreover exactly one element of the set 
\[
\{
\mathbf{a}w_{j}\mathbf{b}w_{i}\mathbf{c}w_{k}\mathbf{d},  \,
\mathbf{a}w_{i}\mathbf{b}w_{k}\mathbf{c}w_{j}\mathbf{d}, \,
\mathbf{a}w_{j}\mathbf{b}w_{k}\mathbf{c}w_{i}\mathbf{d}, \,
\mathbf{a}w_{k}\mathbf{b}w_{i}\mathbf{c}w_{j}\mathbf{d}, \,
\mathbf{a}w_{k}\mathbf{b}w_{j}\mathbf{c}w_{i}\mathbf{d}
\}
\]
belongs to the same excedance class as $w$.
\end{lem}
\begin{proof}
By assumption, $w$ has a $321$-pattern, which we will denote by $i_{0} < j_{0} < k_{0}$.
For the first claim, we proceed in cases based on $i_{0}$ and $k_{0}$.  
If $i_{0} \in \EP(w)$, then either $k_{0} \notin \EP(w)$, in which case we have the desired $321$-pattern, or $k_{0} \in \EP(w)$, in which case we have $i_{0} < j_{0} < k_{0} \le w_{k_{0}} < w_{j_{0}} < w_{i_{0}}$.  
Proceeding with the assumption that $k_{0} \in \EP(w)$, the set $w^{-1}([k_{0}-1])$ does not contain $i_{0}$, $j_{0}$, or $k_{0}$, so there is at least one $k > k_{0}$ with $w_{k} < k_{0}$, and $i_{0} < j_{0} < k$ is a $321$-pattern for which $i_{0} \in \EP(w)$ and $k \notin\EP(w)$.  
On the other hand, if $i_{0} \notin \EP(w)$ then $w_{k_{0}} < w_{j_{0}} < w_{i_{0}} < i_{0} < j_{0} < k_{0}$, so that $k_{0} \notin \EP(w)$.  
In this case $w([i_{0} - 1])$ does not contain $w_{i_{0}}$, $w_{j_{0}}$, or $w_{k_{0}}$, so there is at least one $i < i_{0}$ for which $w_{i} > i_{0}$, and $i < j_{0} < k_{0}$ is a $321$-pattern for which $i \in \EP(w)$ and $k_{0} \notin\EP(w)$.  

For the second claim, we write $i < j < k$ for the $321$-pattern described in the first claim.  There are two cases, depending on whether $j$ is an excedance position or not, and we will only consider the first case, as the second follows from a similar argument.  Assuming that $j \in \EP(w)$, we first show that with $w' = \mathbf{a}w_{j}\mathbf{b}w_{i}\mathbf{c}w_{k}\mathbf{d}$, 
\[
\EP(w') = \EP(w)
\qquad\text{and}\qquad
\EV(w') = \EV(w).
\]
For all $s \in [n] \setminus \{i, j\}$, we have $w'_{i} = w_{i}$, so the above statement reduces to $i, j \in \EP(w')$ and $w_{i}, w_{j} \in \EV(w')$; to see this, observe that $i < j \le w_{j} < w_{i}$, and so $w'_{i} = w_{j} > i$ and $w'_{j} = w_{i} > j$.  
To complete the proof, we verify that each element $w''$ in the set
\[
\{
\mathbf{a}w_{i}\mathbf{b}w_{k}\mathbf{c}w_{j}\mathbf{d}, \,
\mathbf{a}w_{j}\mathbf{b}w_{k}\mathbf{c}w_{i}\mathbf{d}, \,
\mathbf{a}w_{k}\mathbf{b}w_{i}\mathbf{c}w_{j}\mathbf{d}, \,
\mathbf{a}w_{k}\mathbf{b}w_{j}\mathbf{c}w_{i}\mathbf{d}
\}
\]
belongs to a difference excedance class than $w$: either $k \le w''_{k}$, in which case $\EP(w'') \neq \EP(w)$, or $w''_{k} \in \{w_{i}, w_{j}\}$ is a not excedance value for $w''$, in which case $\EV(w'') \neq \EV(w)$.  
\end{proof}

\begin{proof}[Proof of Proposition~\ref{prop:TLbases}]
We proceed by induction on the Bruhat order of $S_{n}$.  If $w$ is $321$-avoiding, then $w = T_{\lambda}$ and the claim clearly holds.  
If $w$ is not $321$-avoiding, we have $T_{\lambda} < w$, so we assume for the sake of induction that for each $v < w$ the claim holds.  
As $w$ has a $321$-pattern, Equation~\eqref{eq:321relation}, Lemma~\ref{lem:321reducebruhat}, and Lemma~\ref{lem:321excedance} allow us to express
\[
w \equiv w' + \sum_{\substack{w'' \le w \\ w'' \notin \CCC_{\lambda} }} b_{w''} w'' \qquad \text{(mod $\ker(\phi)$)},
\]
where $w' < w$ is an element of $\CCC_{\lambda}$, and each coefficient $b_{w''}$ belongs to $\{1, 0, -1\}$.  
We may now apply the inductive hypothesis to the elements $w'$ and $w''$ in the expression above; by Proposition~\ref{prop:Qdominates}, each $w''$ in the sum above belongs to an excedance class $\CCC_{\mu}$ with $\mu \prec \lambda$, so this gives
\[
w \equiv T_{\lambda} + \sum_{\mu \prec \lambda} a_{\mu}^{w'} T_{\mu} + \sum_{\mu \prec \lambda} \sum_{w'' \in \CCC_{\mu}} \left( b_{w''} T_{\mu} + \sum_{\nu \prec \mu} b_{w''} a_{\nu}^{w''} T_{\nu}\right)
\]
for some coefficients $a_{\mu}^{w'}, a_{\mu}^{w''} \in \ZZ$.  Combining like terms, the proof is complete.
\end{proof}

\section{The quasisymmetric variety}
\label{sec:vanishingtheorems}

In this section we prove our final main result, Theorem~\ref{thm:vanishingQSV}.  
The proof rests on a number of intermediate technical results, which we summarize before stating the theorem below.  
The complete details or each intermediate step are given in subsequent subsections.

As in the introduction, let $\QSym_{n}$ denote the quasisymmetric polynomials in $R_{n} = \QQ[x_{2}, \ldots, x_{n}]$ and write $M_{\alpha}$ for the monomial quasisymmetric function indexed by the composition $\alpha$.  In Section~\ref{sec:QSymvanish}, we define a family of non-homogeneous polynomials $P_{\alpha}$ which are also indexed by compositions and we show that
\begin{equation}
\label{eq:Palphahomogeneous}
P_{\alpha} = M_{\alpha} + \text{lower degree terms}.
\end{equation}
For a permutation $\sigma \in S_{n}$, we write $P_{\alpha}(\sigma)$ for the evaluation of $P_{\alpha}$ at $x_{1} = \sigma_{1}$, $x_{2} = \sigma_{2}$, and so on.  Recall the set $\QSV_{n}$ defined in Section~\ref{sec:QSV}.

\begin{thm}
\label{thm:vanishing}
For each non-empty integer composition $\alpha$ with at most $n$ parts and any $\sigma\in \QSV_n$ we have $P_{\alpha} (\sigma)=0$.
\end{thm}

We will assume Theorem~\ref{thm:vanishing} until Section~\ref{sec:QSymvanish}, which contains a proof of the theorem.

Now recall that for any $f \in R_{n}$, $\mathsf{h}(f)$ denotes the homogeneous top-degree component of $f$, and that for any ideal $I \subseteq R_{n}$, we write
\[
\mathsf{gr}(I) = \langle \mathsf{h}(f) \;|\; f \in I \rangle.
\]
In Section~\ref{sec:Grobner}, we show how standard results in Gr\"{o}bner basis theory produce a linear isomorphism $R_{n}/I \cong R_{n} / \mathsf{gr}(I)$, which we also assume in order to prove the following result.

\begin{thm}\label{thm:vanishingQSV} 
The ideal $\langle P_{\alpha} \;|\; \text{non-empty compositions $\alpha$ of length $\ell(\alpha) \le n$} \rangle \subseteq R_n$ is the vanishing ideal $\mathbf{I}(\QSV_n)$ and 
\[
\langle \QSym_{n}^{+} \rangle = \mathsf{gr}\big(\mathbf{I}(\QSV_{n})\big),
\]
where  $\QSym_{n}^{+}$ denotes the set of positive-degree quasisymmetric functions.
 \end{thm}
 
\begin{proof}
Throughout the proof, write
\[
I_{n} = \langle P_{\alpha} \;|\; \text{non-empty compositions $\alpha$ of length $\ell(\alpha) \le n$} \rangle.
\]
We consider the dimension of the quotient of $R_{n}$ by each of $I_{n}$, $\mathbf{I}(\QSV_{n})$, and $\langle \QSym_{n}^{+} \rangle$.

First, the result~\cite[Theorem 1.1]{ABB} states that $\dim(R_n\big/\langle \QSym_{n}^{+} \rangle)$ is the $n$th Catalan number $C_{n} = \frac{1}{n+1} \binom{2n}{n}$.  Equation~\eqref{eq:Palphahomogeneous} shows that $\mathsf{h}(P_{\alpha}) = M_{\alpha}$ for all compositions $\alpha$, so that we have
\[
\langle \QSym_{n}^{+} \rangle \subseteq \mathsf{gr}(I_{n}),
\]
and thus $ \dim(R_n\big/\langle \QSym_{n}^{+} \rangle) \ge \dim(R_n\big/\mathsf{gr}(I_n))$.  
Gr\"{o}bner basis theory then gives a linear isomorphism
\[
R_{n} / \mathsf{gr}(I_{n}) \cong R_{n}/I_{n},
\]
so the dimensions of the two quotients agree.  
Furthermore, using Theorem~\ref{thm:vanishing}, $P_{\alpha} \in \mathbf{I}(\QSV_{n})$ for each composition $\alpha$, so that
\[
I_{n} \subseteq \mathbf{I}(\QSV_{n}),
\]
and consequently $ \dim(R_n\big/I_n) \ge \dim(R_n\big/\mathbf{I}(\QSV_{n})) $.
Finally, because $\mathbf{I}(\QSV_{n})$ is the vanishing ideal for a finite set of points, $\dim(R_{n}\big/\mathbf{I}(\QSV_{n}))$ is also the size of the set $|\QSV_{n}| = C_{n}$.  We therefore have the sequence of inequalities
\[
C_{n} = \dim(R_n\big/\langle \QSym_{n}^{+} \rangle) \ge \dim(R_n\big/\mathsf{gr}(I_n)) =  \dim(R_n\big/I_n) \ge \dim(R_n \big/ \mathbf{I}(\QSV_{n})) = C_{n}.
\]
Therefore, all of the dimensions above are equal to $C_{n}$.  With the containments $\langle \QSym_{n}^{+} \rangle \subseteq \mathsf{gr}(I_{n})$ and $I_{n} \subseteq \mathbf{I}(\QSV_{n})$, this competes the proof.
\end{proof}

Using Gr\"{o}bner basis theory again, we obtain the following corollary.

\begin{cor}
We have $R_n\big/\langle \QSym_{n}^{+} \rangle \cong R_n\big/\mathbf{I}(\QSV_{n})$ as vector spaces.
\end{cor}

\begin{rem}
The Tamari lattice, defined on binary trees or any associated combinatorial objects, is a well-known order (see~\cite{Tamari}). It can also be defined on noncrossing partitions and has the following property: for any noncrossing partition $\lambda$, we define $\ell(\lambda)$ as the length of the shortest path from the identity to $\lambda$ in the Tamari lattice. In a subsequent paper, we will study the polynomial $m_\lambda$ of degree $\ell(\lambda)$, such that for any $Q_\mu\in \QSV_n$, we have
  $$ m_\lambda(Q_\mu)=\begin{cases}1&\text{ if }\lambda=\mu,\\ 0 &\text{ if  } \mu<_w\lambda. \end{cases}$$
Furthermore, the Hilbert series
$$H_n(q) = \sum_{\lambda} q^{\ell(\lambda)}$$
corresponds to the (graded) quotient $R_n\big/\langle \QSym_{n}^{+} \rangle$.
\end{rem}

\begin{rem}
In light of Remark~\ref{rem:altECs}, one should expect the vanishing ideal of the set $\{\omega_{0} w \omega_{0} \;|\; w \in \QSV_{n}\}$ in Remark~\ref{rem:auto} to have similar properties to $\mathbf{I}(\QSV_{n})$, and this turns out to be the case.  In particular, the polynomials $P_\alpha$ can be modified to polynomials $P'_\alpha$ that vanish on every permutation $\omega_{0} w \omega_{0}$ for $w \in \QSV_{n}$. The top homogeneous component of $P'_\alpha$ is also $M_{\alpha}$, leading to an analogue of Theorems~\ref{thm:vanishing} and~\ref{thm:vanishingQSV} for this set.  
This is closely related to the automorphisms of the ring of quasisymmetric functions (see, for example,~\cite{JWY}).
\end{rem}

The remainder of the section fills in the gaps of the proof of Theorem~\ref{thm:vanishingQSV}.  Section~\ref{sec:QSymvanish} gives a complete account of the polynomials $P_{\alpha}$, and Section~\ref{sec:Grobner} describes the isomorphism $R_{n}/\mathsf{gr}(I_{n}) \cong R_{n}/I_{n}$ used in the proof.

\subsection{The vanishing polynomial $P_{\alpha}$}
\label{sec:QSymvanish}

In this section we define the polynomials $P_{\alpha}$ and prove Theorem~\ref{thm:vanishing}.  We begin with a short review of compositions and the refinement order as they relate to $\QSym$.  

A \emph{composition} is a sequence of positive integers $\alpha = (\alpha_{1}, \ldots, \alpha_{k})$.  We refer to $k$ as the \emph{length} of $\alpha$ and to $d = \sum_{i = 1}^{k} \alpha_{i}$ as the \emph{size} of $\alpha$.  Compositions are partially ordered by refinement: the composition $\alpha$ refines another composition $\beta = (\beta_{1}, \ldots, \beta_{\ell})$ if there exists a sequence $1 = f_{1} < f_{2} < \cdots < f_{\ell + 1} = k+1$ for which $\beta_{i} = \alpha_{f_{i}} + \alpha_{f_{i} + 1} + \cdots + \alpha_{f_{i+1} - 1}$, and in this case we write $\beta \succeq \alpha$.  Whenever we have a refinement relation $\beta \succeq \alpha$, we will use the notation $f_{1}, f_{2}, \ldots, f_{\ell+1}$ to refer to the sequence of indices in the definition.

For each composition of length $k \ge 1$, the monomial quasisymmetric function $M_{\alpha} \in R_{n}$ is defined by
\[
M_{\alpha} = \sum_{1\le i_1<i_2<\cdots<i_k\le n} x_{i_1}^{\alpha_1} x_{i_2}^{\alpha_2}\cdots  x_{i_k}^{\alpha_k},
\]
where the sum is over subsets $\{i_{1}, \ldots, i_{k}\}$ of $[n]$, enumerated in increasing order.  Using the same convention we define the vanishing polynomial $P_{\alpha} \in R_{n}$ to be
\[
P_{\alpha} =  
\sum_{\beta \succeq \alpha}
\hspace{0.6em}
\sum_{1\le i_{1} < i_{2} < \cdots < i_{\ell} \le n} 
\hspace{0.6em}
\prod_{j = 1}^{\ell}
\Big(
(x_{i_{j}}^{\alpha_{f_{j}}} - i_{j}^{\alpha_{f_{j}}}) 
\prod_{s = f_{j} + 1}^{f_{j+1} - 1} (- i_{j})^{\alpha_{s}}
\Big).
\]
While this formula appears to be quite dense, expanding it reveals an intuitive combinatorial structure.  We compute one example in its entirety for the sake of exposition: 
\begin{align*}
P_{(1, 2, 1)}(x_{1}, \ldots, x_{4}) =& (x_{1} - 1)(x_{2}^{2} - 2^{2})(x_{3} - 3) + (x_{1} - 1)(x_{2}^{2} - 2^{2})(x_{4} - 4) \\
&+ (x_{1} - 1)(x_{3}^{2} - 3^{2})(x_{4} - 4) + (x_{2} - 2)(x_{3}^{2} - 3^{2})(x_{4} - 4) \\[1ex]
&\hspace{1.2em} -  (x_{1} - 1)(x_{2}^{2} - 2^{2})2 -  (x_{1} - 1)(x_{3}^{2} - 3^{3})3 -  (x_{1} - 1)(x_{4}^{2} - 4^{2})4  \\
&\hspace{1.2em} -  (x_{2} - 2)(x_{3}^{2} - 3^{3})3 -  (x_{2} - 2)(x_{4}^{2} - 4^{2})4 -  (x_{3} - 3)(x_{4}^{2} - 4^{2})4 \\[1ex]
&\hspace{2.4em} -  (x_{1} - 1)1^{2}(x_{2} - 2) -  (x_{1} - 1)1^{2}(x_{3} - 3) -  (x_{1} - 1)1^{2}(x_{4} - 4) \\
&\hspace{2.4em} -  (x_{2} - 2)2^{2}(x_{3} - 3) -  (x_{2} - 2)2^{2}(x_{4} - 4) -  (x_{3} - 3)3^{2}(x_{4} - 4) \\[1ex]
&\hspace{3.6em} + (x_{1} - 1) 1^{3} + (x_{2} - 2) 2^{3} + (x_{3} - 3) 3^{3} + (x_{4} - 4) 4^{3},
\end{align*}
where summands corresponding to the same index $\beta \succeq (1, 2, 1)$ are grouped horizontally and by alignment.  These values of $\beta$ are respectively $(1, 2, 1)$, $(1, 3)$, $(3, 1)$, and $(4)$.

\begin{prop}
For all compositions $\alpha$, $\mathsf{h}(P_{\alpha}) = M_{\alpha}$.
\end{prop}

\begin{proof}
The top-degree term in the summand of $P_{\alpha}$ indexed by $\beta \succeq \alpha$ has degree equal to the length $\ell$ of $\beta$, so we may immediately discard all summands except for the $\beta = \alpha$ one.  Expanding this summand, the top-degree terms have the form $x_{i_1}^{\alpha_1} x_{i_2}^{\alpha_2}\cdots  x_{i_k}^{\alpha_k}$, with one such term for each subset of indices $1\le i_{1} < i_{2} < \cdots < i_{k} \le n$.
\end{proof}

This proves Equation~\eqref{eq:Palphahomogeneous} from the previous section.  We now begin our proof of Theorem~\ref{thm:vanishing}, which states that $P_{\alpha}(\sigma) = 0$ for each composition $\alpha$ and permutation $\sigma \in \QSV_{n}$.  We first prove a special case, with the full proof coming after the proof of Lemma~\ref{lem:onecyclevanishing}.

\begin{lem}
\label{lem:onecyclevanishing}
Suppose that $\sigma \in \QSV_{n}$ has at most one cycle of length greater than $1$.  Then for all compositions $\alpha$, $P_{\alpha}(\sigma) = 0$.
\end{lem}
\begin{proof}
We begin with a simple but significant reduction using the observation that for any exponent $e$, the expression $(x_{i}^{e} - i^{e})$ vanishes at $x_{i} = i$.  If $\sigma$ is the identity, the claim follows from this observation and the definition of $P_{\alpha}$.  
Assuming the contrary, write $C = \{a_{1} < a_{2} < \cdots < a_{m}\}$ for the elements of the unique long cycle in $\sigma$.  By Lemma~\ref{lem:QSVcycles}, we know that $\sigma = (a_{m}\;\ldots\;a_{2}a_{1})$.  We then have
\begin{equation}
\label{eq:Patsigma1}
P_{\alpha}(\sigma) =  
\sum_{\beta \succeq \alpha}
\hspace{0.6em}
\sum_{\substack{ 1\le i_{1} < i_{2} < \cdots < i_{\ell} \le n \\ \text{$i_{j} \in C$ for all $j \in [\ell]$} }}
\hspace{0.6em}
\prod_{j = 1}^{\ell}
\Big(
(\sigma_{i_{j}}^{\alpha_{f_{j}}} - i_{j}^{\alpha_{f_{j}}}) 
\prod_{s = f_{j} + 1}^{f_{j+1} - 1} (- i_{j})^{\alpha_{s}},
\Big)
\end{equation}
where the second sum is over indices $1\le i_{1} < i_{2} < \cdots < i_{\ell} \le n$ not fixed by $\sigma$.  

For the remainder of the proof, we adopt the convention that $a_{0} = a_{m}$, so that $\sigma_{a_{i}} = a_{i-1}$ for each $1 \le i \le m$.  We now expand the formula for $P_{\alpha}(\sigma)$ above, from which we obtain
\begin{equation}\label{eq:Patsigma2}
	P_{\alpha}(\sigma)=
	\sum_{{1\le j_1\le j_2\le\cdots\le j_k\le m}} \sum_{{\epsilon_1,\epsilon_2,\ldots,\epsilon_k \atop \epsilon_i\in \{0,1\}}\atop \epsilon_i=1 \text{ if } j_{i-1}=j_i} (-1)^{\sum \epsilon_i}\ 
	   z_{j_1,\epsilon_1}^{\alpha_1}z_{j_2,\epsilon_2}^{\alpha_2} \cdots z_{j_k,\epsilon_k}^{\alpha_k}\,,
\end{equation}
where we write $z_{j_i,\epsilon_i} = a_{j_{i} - 1 + \epsilon_{i}}$, so that
\[
z_{j_i,\epsilon_i} = \begin{cases} a_{j_i-1} & \text{if $\epsilon_i=0$,} \\ a_{j_i} & \text{if $\epsilon_{i} = 1$.}  \end{cases} 
\]

We complete the proof by showing that Equation~\eqref{eq:Patsigma2} is equal to zero.  We do by way of a sign-reversing involution on the summands of the right side of the equation.  To this end, fix two sequences $1\le j_1\le j_2\le\cdots\le j_k\le m$ and $\epsilon_1,\epsilon_2,\ldots,\epsilon_k \in \{0, 1\}$ for which $\epsilon_{i} = 1$ whenever $j_{i-1} = j_{i}$, as in the sum.  

We now take $d$ be be the maximal index for which $z_{j_{d}, \epsilon_{d}}$ is neither $0$ nor $m$, if such an index exists, or $d = 1$ otherwise.  We define two new sequences $1\le j_1'\le j_2'\le\cdots\le j_k'\le m$ and $\epsilon_1',\epsilon_2',\ldots,\epsilon_k' \in \{0,  1\}$ by stipulating that $j_{i}' = j_{i}$ and $\epsilon_{i}' = \epsilon_{i}$ if $i \neq d$, along with the following formulas:
\[
j_{d}' = \begin{cases} 
j_{d} + 1 & \text{if $\epsilon_{d} = 1$ and $j_{d} \neq n$} \\ 
1 & \text{if $\epsilon_{d} = 1$ and $j_{d} = n$} \\
j_{d} - 1 & \text{if $\epsilon_{d} = 0$ and $j_{d} \neq 1$} \\
n & \text{if $\epsilon_{d} = 0$ and $j_{d} = 1$}
\end{cases}
\qquad\text{and}\qquad
\epsilon_{d}' = \begin{cases}
1 & \text{if $\epsilon_{d} = 0$} \\ 
0 & \text{if $\epsilon_{d} = 1$.} \\ 
\end{cases}
\]
From this definition, one can verify that $z_{j_{i}, \epsilon_{i}} = z_{j_{i}', \epsilon_{i}'}$, as in all cases 
\[
j_{i} - 1 + \epsilon_{i} \equiv j_{i}' - 1 + \epsilon_{i}' \qquad \text{(mod $n$)}.
\]
This also shows that the operation described above is an involution: it does not change the value of $d$, and the for a fixed $d$ the description above is clearly an involution.  Finally, 
\[
(-1)^{\sum \epsilon'_i}\ z_{j'_1,\epsilon'_1}^{\alpha_1}z_{j'_2,\epsilon'_2}^{\alpha_2} \cdots z_{j'_k,\epsilon'_k}^{\alpha_k}
  = - (-1)^{\sum \epsilon_i}\ z_{j_1,\epsilon_1}^{\alpha_1}z_{j_2,\epsilon_2}^{\alpha_2} \cdots z_{j_k,\epsilon_k}^{\alpha_k}.
\]
Thus, all terms in $P_{\alpha}(\sigma)$ cancel with their image under this involution, completing the proof.
\end{proof}

In the remainder of the section, we will show how the general case follows from Lemma~\ref{lem:onecyclevanishing}.  To begin, let $\lambda$ be a noncrossing partition and fix a subset $\{i_{1}, i_{2}, \ldots, i_{\ell}\}$ of $[n]$.  
Define the \emph{leftmost nested component of $\lambda$ with respect to $\{i_{1}, i_{2}, \ldots, i_{\ell}\}$} as the leftmost connected component of $\lambda$ which contains some element $i_{j}$ and furthermore has no other connected component with this property ``nested'' between its arcs.  Formally, if we write $C$ for this component, $C$ is defined by the properties 
\begin{enumerate}[label = (\roman*), itemsep = 1ex]
\item $C \cap \{i_{1}, i_{2}, \ldots, i_{\ell}\} \neq \emptyset$,

\item for all $j \in [\ell]$ with $\min(C) \le i_{j} \le \max(C)$, we have $i_{j} \in C$, and

\item if $D \neq C$ is a connected component satisfying (i) and (ii), then $\max(C) < \min(D)$.

\end{enumerate}

For example, using the noncrossing partition
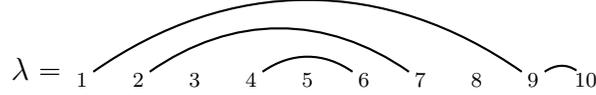
\[
\lambda = \begin{tikzpicture}[scale = 0.75, baseline = 0.75*-0.2]
\foreach \x in {1, ..., 9}{\draw[fill] (\x - 1, 0) node[inner sep = 2pt] (\x) {$\scriptstyle \x$};}
\draw[fill] (10 - 1, 0) node[inner sep = 2pt] (10) {$\hspace{-0.25em} \scriptstyle 10$};
\foreach \i\j in { 4/6, 2/7,1/9, 9/10}{\draw[thick] (\i) to[out = 35, in = 145] (\j);}
\end{tikzpicture}
\]
the leftmost nested component of $\lambda$ with respect to the set $\{1, 3, 4, 5, 8\}$ is $\{5\}$, and with respect to the set $\{1, 2, 8, 10\}$ the leftmost nested component of $\lambda$ is $\{2, 7\}$.

\begin{proof}[Proof of Theorem~\ref{thm:vanishing}]
The noncrossing structure of the cycles of $\sigma$ will be essential in the proof.  
Let $\lambda$ be the noncrossing partition for which $\sigma = Q_{\lambda}$, as defined in Section~\ref{sec:QSV}.  

We will carefully examining the expression
\begin{equation}\label{eq:Palpha_evaluated}
P_{\alpha}(\sigma) =  
\sum_{\beta \succeq \alpha}
\hspace{0.6em}
\sum_{1\le i_{1} < i_{2} < \cdots < i_{\ell} \le n} 
\hspace{0.6em}
\prod_{j = 1}^{\ell}
\Big(
(\sigma_{i_{j}}^{\alpha_{f_{j}}} - i_{j}^{\alpha_{f_{j}}}) 
\prod_{s = f_{j} + 1}^{f_{j+1} - 1} (- i_{j})^{\alpha_{s}}
\Big),
\end{equation}
and show that it can be divided into summands which each have a factor of the form $P_{\alpha'}(\tau)$, where $\alpha'$ is another composition and $\tau \in \QSV_{n}$ is one of the cycles of $\sigma$, taken as a permutation that fixes all other points.  Lemma~\ref{lem:onecyclevanishing} states that each $P_{\alpha'}(\tau) = 0$, so this will complete the proof.


We now analyze $P_{\alpha}(\sigma)$; recall the definition of a leftmost nested component given before the proof.  
Each summand in the formula~\eqref{eq:Palpha_evaluated} corresponds to the choice of a composition $\beta \succeq \alpha$ and a subset $\{i_{1}, i_{2}, \ldots, i_{\ell}\}$ of $[n]$.  
Let $C$ be the leftmost nested component of $\lambda$ with respect to this subset, and let $p$ and $q$ count the respective number of indices $i_{j}$ which appear left and right of the component $C$ in $\lambda$.  
Thus, $i_{p} < \min(C)$, $i_{\ell - q+1} > \max(C)$, and $\{i_{p+1}, \ldots, i_{\ell - q}\} \subseteq C$.  
Record the specifics of this configuration with four additional sets:
\begin{align*}
F_{\text{left}} &= \{f_{1}, \ldots, f_{p}, f_{p+1}\}, \\
F_{\text{right}} &= \{f_{\ell - q+1}, \ldots, f_{\ell+1}\}, \\
I_{\text{left}} &= \{i_{1}, \ldots, i_{p}\},\;\text{and} \\ 
I_{\text{right}} &= \{i_{\ell - q+1}, \ldots, i_{\ell}\}.
\end{align*}  
The assignment of each summand to a tuple $(C, F_{\text{left}}, F_{\text{right}}, I_{\text{left}}, I_{\text{right}})$ is uniquely determined, and therefore gives a partition of the summands of $P_{\alpha}(\sigma)$.  Using this partition, we write
\[
P_{\alpha}(\sigma) = \sum_{\text{tuples}} P_{\alpha}^{(C, F_{\text{left}}, F_{\text{right}}, I_{\text{left}}, I_{\text{right}})}
\]
where $P_{\alpha}^{(C, F_{\text{left}}, F_{\text{right}}, I_{\text{left}}, I_{\text{right}})}$ denotes the sum of all terms in $P_{\alpha}(\sigma)$ which determine the tuple $(C, F_{\text{left}}, F_{\text{right}}, I_{\text{left}}, I_{\text{right}})$ in the above process, and the sum is over all tuples which correspond to some summand of $P_{\alpha}(\sigma)$.  

We now show that each summand $P_{\alpha}^{(C, F_{\text{left}}, F_{\text{right}}, I_{\text{left}}, I_{\text{right}})}$ above factors in a predictable way.  
Respectively enumerate the elements of $F_{\text{left}}$, $F_{\text{right}}$, $I_{\text{left}}$, and $I_{\text{right}}$ in increasing order as
\begin{align*}
F_{\text{left}} &= \{\underline{f}_{1}, \ldots, \underline{f}_{p}, \underline{f}_{p+1}\}, \\
F_{\text{right}} &= \{\overline{f}_{1}, \overline{f}_{2}, \ldots, \overline{f}_{q+1} \}, \\
I_{\text{left}} &= \{\underline{i}_{1}, \ldots, \underline{i}_{p}\},\;\text{and} \\ 
I_{\text{right}} &= \{\overline{i}_{1}, \ldots, \overline{i}_{q}\}.
\end{align*}  
In order for these sets to come from an actual summand of $P_{\alpha}(\sigma)$, it must be the case that
\[
1 = \underline{f}_{1} < \cdots < \underline{f}_{p} < \underline{f}_{p + 1} < \overline{f}_{1} < \overline{f}_{2} < \cdots < \overline{f}_{q+1} = k + 1,
\]
and 
\[
1\le \underline{i}_{1} < \cdots \underline{i}_{p} < \min(C) \le \max(C) < \overline{i}_{1} < \cdots < \overline{i}_{q} \le n.
\]  
In fact, each summand of $P_{\alpha}(\sigma)$ which appears in $P_{\alpha}^{(C, F_{\text{left}}, F_{\text{right}}, I_{\text{left}}, I_{\text{right}})}$ corresponds to a composition $\beta'$ which is refined by 
\[
\alpha' = (\alpha_{\underline{f}_{p + 1}}, \ldots, \alpha_{\overline{f}_{1} -1})
\]
and to a set of indices $1 \le i_{1} < \cdots < i_{\ell'} \le n$ contained in $C$, where $\ell'$ is the length of $\beta'$.  This correspondence is one-to-one, and we have
\begin{align*}
P_{\alpha}^{(C, F_{\text{left}}, F_{\text{right}}, I_{\text{left}}, I_{\text{right}})} =& \left( 
\prod_{j = 1}^{p}
\Big(
(\sigma_{\underline{i}_{j}}^{\alpha_{\underline{f}_{j}}} - i_{j}^{\alpha_{\underline{f}_{j}}}) 
\prod_{s = \underline{f}_{j} + 1}^{\underline{f}_{j+1} - 1} (- \underline{i}_{j})^{\alpha_{s}}
\Big)  
\right) \\ 
& \hspace{2em} \left( 
\prod_{j = 1}^{q}
\Big(
(\sigma_{\overline{i}_{j}}^{\alpha_{\overline{f}_{j}}} - \overline{i}_{j}^{\alpha_{\overline{f}_{j}}}) 
\prod_{s = \overline{f}_{j} + 1}^{\overline{f}_{j+1} - 1} (- \overline{i}_{j})^{\alpha_{s}}
\Big) 
\right) \\
& \hspace{4em} \left(
\sum_{\beta' \succeq \alpha'} 
\hspace{0.6em}
\sum_{\substack{ 1 \le i_{1} < i_{2} < \cdots < i_{\ell'} \le n\\ \text{$i_{j} \in C$ for all $j \in [\ell']$} }} 
\hspace{0.6em}
\prod_{j = 1}^{\ell'}
\Big(
(\sigma_{i_{j}}^{\alpha_{f_{j}}} - i_{j}^{\alpha_{f_{j}}}) 
\prod_{s = f_{j} + 1}^{f_{j+1} - 1} (- i_{j})^{\alpha_{s}}
\right)
\end{align*}
where the numbers $f_{j}$ appearing in the final factor above correspond the the refinement relation $\beta' \succeq \alpha'$.  

In conclusion, let $\tau$ be the permutation defined by $\tau_{i} = \sigma_{i}$ for $i \in C$ and $\tau_{i} = i$ for all other $i \in [n]$.  
Then $\tau \in \QSV_{n}$ and $\tau$ has a single cycle of length more than $1$.  
Furthermore, the final factor of $P_{\alpha}^{(C, F_{\text{left}}, F_{\text{right}}, I_{\text{left}}, I_{\text{right}})}$ above is exactly the expression for $P_{\alpha'}(\tau)$ given in Equation~\eqref{eq:Patsigma1}.  
Lemma~\ref{lem:onecyclevanishing} states that $P_{\alpha'}(\tau) = 0$, so this completes the proof.
\end{proof}

\subsection{Gr\"{o}bner basis theory and top-degree homogeneous ideals}
\label{sec:Grobner}

In this section we use Gr\"{o}bner basis theory to show that for all ideals $I \subseteq R_{n}$, there is a vector space isomorphism
\[
R_{n} / I \cong R_{n}/\mathsf{gr}(I).
\]
This result and the others in this section are not novel and can be found in many standard texts (for example, see~\cite{CLO}).
However, we include a streamlined account here so that our proof of Theorem~\ref{thm:vanishing} is more readily accessible to a wide audience.  

For any $f \in R_{n}$, let $\mathsf{lt}(f)$ denote the leading term of $f$ with respect to the graded lexicographic term order.  For an ideal $I \subseteq R_{n}$, define the \emph{initial ideal} of $I$ to be 
\[
\mathsf{in}(I) = \langle \mathsf{lt}(f) \;|\; f \in I \rangle.
\]

\begin{prop}
\label{prop:grin}
Let $I$ be an ideal of $R_{n}$.  Then $\mathsf{in}(\mathsf{gr}(I)) = \mathsf{in}(I)$.
\end{prop}
\begin{proof}
It is straightforward to verify that for any $f \in I$ we have $\mathsf{lt}(f) = \mathsf{lt}(\mathsf{h}(f))$, and therefore $\mathsf{in}(\mathsf{gr}(I)) \supseteq \mathsf{in}(I)$. To show the other direction of containment, note that
\[
\mathsf{gr}(I) = \QQ\spanning\{\mathsf{h}(f) \;|\; f \in I\}.
\]
Now fix a monomial $x_{1}^{a_{1}} x_{2}^{a_{2}}\cdots x_{n}^{a_{n}} \in \mathsf{in}(\mathsf{gr}(I))$, so that
\[
x_{1}^{a_{1}}  x_{2}^{a_{2}} \cdots x_{n}^{a_{n}} = \mathsf{lt}\left(\sum_{i =1}^{k} c_{i} \mathsf{h}(f_{i}) \right)
\]
for some $f_{1}, \ldots, f_{k} \in I$ and $c_{1}, \ldots, c_{n} \in \QQ$.  As our term order respects degree, we assume without loss of generality that each $f_{i}$ has degree $a_{1} + a_{2} + \cdots + a_{n}$.  Therefore 
\[
\left(\sum_{i =1}^{k} c_{i} \mathsf{h}(f_{i}) \right) = \mathsf{h}\left(\sum_{i =1}^{k} c_{i} f_{i} \right),
\]
and together with our previous remarks this completes the proof.
\end{proof}

The desired isomorphism follows from a sequence of isomorphisms $R_{n}/\mathsf{gr}(I) \cong R_{n}/\mathsf{in}(I) \cong R_{n}/I$ which are guaranteed to exist by Proposition~\ref{prop:grin} the next result, which is a classical statement in Gro\"{o}bner basis theory.  

\begin{prop}
\label{prop:Gbasis}
For all ideals $I$ of $R_{n}$, there is a linear isomorphism $R_{n}/I \cong R_{n} / \mathsf{in}(I)$.  In particular, the monomials in $R_{n} \setminus \mathsf{in}(I)$ descend to a linear basis of both quotients.
\end{prop}
\begin{proof}
Write $M$ for the set of monomials in $R_{n}$ which are note contained in $\mathsf{in}(I)$.  To see that $M$ descends to a basis of $R_{n} / \mathsf{in}(I)$, note that the set of all monomials is a basis of $R_{n}$ and 
\[
\mathsf{in}(I) = \QQ\spanning\{\mathsf{lt}(f) \;|\; f \in I\}.
\]

To show the analogous statement $R_{n}/I$, we first remark the existence of a Gr\"{o}bner basis of $I$: a generating set $g_{1}, \ldots, g_{k}$ of $I$ for which $\langle \mathsf{lt}(g_{1}), \mathsf{lt}(g_{2}), \ldots, \mathsf{lt}(g_{k}) \rangle = \mathsf{lt}(I)$.
For any $f \in R_{n}$, applying the division algorithm with this Gr\"{o}bner basis will produce a representative of $f$ (mod $I$) which is a linear combination of the elements of $M$, so we have a spanning set.  
Moreover, $M$ (mod $I$) is linearly independent: the leading term of any nonzero linear combination from $M$ will belong to $M$, so any such linear combination cannot be contained in $I$.  
Therefore $M$ descends to a basis of the quotient $R/I$.
\end{proof}

\begin{rem}
Using the method of Proposition~\ref{prop:Gbasis},~\cite[Theorem 4.1]{ABB} gives an explicit basis for the quotient $R_{n}/\langle \QSym_{n}^{+}\rangle$, and therefore also for $R_{n}/\mathbf{I}(\QSV_{n})$.
\end{rem}

\end{document}